\documentclass[11pt]{amsart}

\usepackage{amsfonts,amssymb,amsmath,amsthm,latexsym,graphics,epsfig}
\usepackage{verbatim,enumerate,array,booktabs,color,bigstrut,prettyref,tikz-cd}
\usepackage{multirow}
\usepackage[all]{xy}
\usepackage[backref]{hyperref}
\usepackage[OT2,T1]{fontenc}
\usepackage{ctable}

\usepackage{mathtools}

\newrefformat{eq}{\textup{(\ref{#1})}}
\newrefformat{prty}{\textup{(\ref{#1})}}

\definecolor{mylinkcolor}{rgb}{0.8,0,0}
\definecolor{myurlcolor}{rgb}{0,0,0.8}
\definecolor{mycitecolor}{rgb}{0,0,0.8}
\hypersetup{colorlinks=true,urlcolor=myurlcolor,citecolor=mycitecolor,linkcolor=mylinkcolor,linktoc=page,breaklinks=true}

\DeclareSymbolFont{cyrletters}{OT2}{wncyr}{m}{n}
\DeclareMathSymbol{\Sha}{\mathalpha}{cyrletters}{"58}

\newtheorem{defn}{Definition}[section]
\newtheorem{definition}[defn]{Definition}

\newtheorem{corollary}[defn]{Corollary}
\newtheorem{lemma}[defn]{Lemma}

\newtheorem{thm}[defn]{Theorem}
\newtheorem{theorem}[defn]{Theorem}

\newtheorem{proposition}[defn]{Proposition}

\theoremstyle{definition}

\newtheorem*{ack}{Acknowledgements}
\newtheorem{remark}[defn]{Remark}

\newcommand{\QQ}{\mathbb Q}

\newcommand{\ZZ}{\mathbb Z}
\newcommand{\Z}{\mathbb Z}
\newcommand{\Q}{\mathbb{Q}}

\newcommand{\PP}{\mathbb P}

\newcommand{\arrow}{\longrightarrow}

\newcommand{\tor}{\mathrm{tors}}

\begin{document}
	
	
	
	\title[A classification of curious Galois groups as direct products]{A classification of curious Galois groups as direct products}
	
\author{Garen Chiloyan}
\email{garen.chiloyan@gmail.com} 
\urladdr{https://sites.google.com/view/garenmath/home}




\begin{abstract}
Let $N$ be a positive integer. Let $\operatorname{H}$ be a group of level $N$ and let $E$ be an elliptic curve defined over the rationals with $\textit{j}_{E} \neq 0, 1728$. Then the image $\overline{\rho}_{E,N}\left(\operatorname{Gal}\left(\overline{\QQ}/\QQ\right)\right)$, of the mod-$N$ Galois representation attached to $E$, 
is conjugate to a subgroup of $\operatorname{H}$ if and only if $E$ corresponds to a non-cuspidal rational point on the modular curve $\operatorname{X}_{\operatorname{H}}$ generated by $\operatorname{H}$. In this article, we are interested when $\overline{\rho}_{E,N}(G_{\QQ})$ is precisely $\operatorname{H}$. More precisely, we classify all groups $\operatorname{H}$ that are direct products of subgroups $\operatorname{H}$ for which $\operatorname{X}_{\operatorname{H}}$ contains infinitely many non-cuspidal rational points but there is no elliptic curve $E/\QQ$ such that $\overline{\rho}_{E,N}\left(\operatorname{Gal}\left(\overline{\QQ}/\QQ\right)\right)$ is conjugate to $\operatorname{H}$.
\end{abstract}

\maketitle

\section{Introduction}

In Section 6 of \cite{Rouse}, Rouse and Zureick-Brown found a group $\operatorname{H}_{155}$, of level $16$ that generates a modular curve $\operatorname{X}_{155}$. The curve $\operatorname{X}_{155}$ is an elliptic curve with Mordell--Weil group isomorphic to $\ZZ \times \ZZ / 2 \ZZ$. There are infinitely many elliptic curves $E/\QQ$ up to isomorphism such that $\overline{\rho}_{E,16}(G_{\QQ})$ is conjugate to a proper subgroup of $\operatorname{H}_{155}$ and yet, there is no elliptic curve $E'/\QQ$ such that $\overline{\rho}_{E',16}(G_{\QQ})$ is conjugate to $\operatorname{H}_{155}$ precisely. The authors of \cite{Rouse} referred to groups like $\operatorname{H}_{155}$ as ``curious'' examples. The authors of \cite{Rouse} found seven ``curious'' examples of $2$-power-level in total, namely, $\operatorname{H}_{150}$, $\operatorname{H}_{153}$, $\operatorname{H}_{155}$, $\operatorname{H}_{156}$, $\operatorname{H}_{165}$, $\operatorname{H}_{166}$, and $\operatorname{H}_{167}$ in the labeling of the database (RZB) in \cite{Rouse}. The modular curves $\operatorname{X}_{\operatorname{H}}$ that are generated by the seven ``curious'' examples of $2$-power-level, $\operatorname{H}$, are elliptic curves with Mordell--Weil group isomorphic to $\ZZ \times \ZZ / 2 \ZZ$. Moreover, in each of the cases, the rational points on $\operatorname{X}_{\operatorname{H}}$ lift to covering modular curves. For the purposes of this paper, we define arithmetically admissible groups of level $N$ and curious Galois groups of level $N$.

\begin{definition}\label{Def 1}
Let $N$ be a positive integer and let $\operatorname{H}$ be a subgroup of $\operatorname{GL}(2, \ZZ / N \ZZ)$ of level $N$. Then $\operatorname{H}$ is an arithmetically admissible group of level $N$ if the following three statements are satisfied:
\begin{enumerate}
    \item $\operatorname{-Id} \in \operatorname{H}$,

    \item $\operatorname{Det}(\operatorname{H}) = \left(\ZZ / N \ZZ\right)^{\times}$,

    \item $\operatorname{H}$ contains an element that is conjugate over $\operatorname{GL}(2, \ZZ / N \ZZ)$ to $\begin{bmatrix} 1 & 0 \\ 0 & -1 \end{bmatrix}$ or $\begin{bmatrix}
        1 & 1 \\ 0 & -1
    \end{bmatrix}$
\end{enumerate}
\end{definition}

One would expect that if $\operatorname{X}_{\operatorname{H}}$ is a modular curve generated by a group $\operatorname{H}$ such that $\operatorname{X}_{\operatorname{H}}$ has infinitely many rational points, then there would be at least one elliptic curve $E/\QQ$ such that $\overline{\rho}_{E,N}(G_{\QQ})$ is conjugate to $\operatorname{H}$. This is not the case and with this failure, there exists many non-examples of sorts to a possible extension of Hilbert's irreducibility theorem in the case when the genus of the modular curve is equal to $1$.

\begin{definition}\label{Def 2}
Let $\operatorname{H}$ be an arithmetically admissible group of level $N$ and denote the modular curve defined by $\operatorname{H}$ by $\operatorname{X}_{\operatorname{H}}$. Then $\operatorname{H}$ is a curious Galois group of level $N$ if both of the following statements are satisfied.

\begin{enumerate}
    \item $\operatorname{X}_{\operatorname{H}}(\QQ)$ is infinite.
    \item There is no elliptic curve $E/\QQ$ such that $\overline{\rho}_{E,N}(G_{\QQ})$ is conjugate to $\operatorname{H}$.
\end{enumerate}
\end{definition}

In section 5 of \cite{Daniels2018SerresCO}, Daniels and Gonzalez-Jimenez found a single curious Galois group of level $24$ and a single curious Galois group of level $15$ which they denoted $[\texttt{8X5, 3Nn}]$ and $[\texttt{3Nn, 5S4}]$, respectively. By (unknowingly) using similar techniques from \cite{Daniels2018SerresCO}, the author of this paper found another curious Galois group of level $24$, namely $\mathcal{C}_{sp}(3) \times \operatorname{H}_{5}$, in section 9 of \cite{Chiloyan20232adicGI}. We will refer to groups (when available) and elliptic curves by their label from \cite{rouse_sutherland_zureick-brown_2022} and hyperlink them to the LMFDB \cite{lmfdb}. Work in Subsection 2.4 of \cite{rouse_sutherland_zureick-brown_2022} explains that all groups $\operatorname{H}$ in the LMFDB have a label of the form \texttt{N.i.g.n} where \texttt{N}, \texttt{i}, \texttt{g}, and \texttt{n}, are the decimal representations of the integers $N$, $i$, $g$, $n$ defined as follows:

\begin{itemize}
    \item $N = N(\operatorname{H})$ is the level of $\operatorname{H}$;
    \item $i = i(\operatorname{H})$ is the index of $\operatorname{H}$ as a subgroup of $\operatorname{GL}\left(2, \widehat{\ZZ}\right)$;
    \item $g = g(\operatorname{H})$ is the genus of the modular curve generated by $\operatorname{H}$ (which we will also call the genus of $\operatorname{H}$);
    \item $n$ is a positive integer and is an ordinal indicating the position of $\operatorname{H}$ among all subgroups of the same level, index, and genus, with $\operatorname{Det}(\operatorname{H}) = \widehat{\ZZ}^{\times}$ under an ordering defined by the authors of \cite{Rouse2021elladicIO}.
    
\end{itemize}

The main theorem and focus of this paper is Theorem \ref{main Theorem}. 
\begin{theorem}\label{main Theorem}

Let $m$ be a positive integer. Let $p_{1} < \ldots < p_{m}$ be a collection of primes and for each $i \in \left\{1, \ldots, m\right\}$, let $\operatorname{H}_{p_{i}}$ be a proper subgroup of $\operatorname{GL}(2, \ZZ_{p_{i}})$. Suppose that $\operatorname{H} = \operatorname{H}_{p_{1}} \times \ldots \times \operatorname{H}_{p_{m}}$ is a curious Galois group. Then $m$ equals $1$ or $2$.

\begin{itemize}
    \item If $m = 1$, then $\operatorname{H}$ is one of the following seven groups:

    \begin{enumerate}
        \item \href{https://lmfdb.org/ModularCurve/Q/16.24.1.5/}{\texttt{16.24.1.5}},
        \item \href{https://lmfdb.org/ModularCurve/Q/16.24.1.10/}{\texttt{16.24.1.10}},
        \item \href{https://lmfdb.org/ModularCurve/Q/16.24.1.11/}{\texttt{16.24.1.11}},
        \item \href{https://lmfdb.org/ModularCurve/Q/16.24.1.13/}{\texttt{16.24.1.13}},
        \item \href{https://lmfdb.org/ModularCurve/Q/16.24.1.15/}{\texttt{16.24.1.15}},
        \item \href{https://lmfdb.org/ModularCurve/Q/16.24.1.17/}{\texttt{16.24.1.17}},
        \item \href{https://lmfdb.org/ModularCurve/Q/16.24.1.19/}{\texttt{16.24.1.19}}.
    \end{enumerate}

    \item If $m = 2$, then $\operatorname{H}$ is one of the following fifteen groups:
    \begin{enumerate}
        \item \href{https://lmfdb.org/ModularCurve/Q/15.15.1.1/}{\texttt{15.15.1.1}}, the direct product of \texttt{3.3.0.1} and \texttt{5.5.0.1},
        \item the direct product of \texttt{8.2.0.1} and \texttt{13.14.0.1},
        \item the direct product of \texttt{8.2.0.2} and \texttt{9.12.0.1},
        \item \href{https://lmfdb.org/ModularCurve/Q/40.12.1.5/}{\texttt{40.12.1.5}}, the direct product of \texttt{8.2.0.1} and \texttt{5.6.0.1},
        \item \href{https://lmfdb.org/ModularCurve/Q/40.20.1.2/}{\texttt{40.20.1.2}}, the direct product of \texttt{8.2.0.2} and \texttt{5.10.0.1},
        \item \href{https://lmfdb.org/ModularCurve/Q/40.36.1.2/}{\texttt{40.36.1.2}}, the direct product of \texttt{8.6.0.2} and \texttt{5.6.0.1},
        \item \href{https://lmfdb.org/ModularCurve/Q/40.36.1.4/}{\texttt{40.36.1.4}}, the direct product of \texttt{8.6.0.4} and \texttt{5.6.0.1},
        \item \href{https://lmfdb.org/ModularCurve/Q/40.36.1.5/}{\texttt{40.36.1.5}}, the direct product of \texttt{8.6.0.5} and \texttt{5.6.0.1},
        \item \href{https://lmfdb.org/ModularCurve/Q/24.6.1.2/}{\texttt{24.6.1.2}}, the direct product of \texttt{8.2.0.2} and \texttt{3.3.0.1},
        \item \href{https://lmfdb.org/ModularCurve/Q/24.12.1.3/}{\texttt{24.12.1.3}}, the direct product of \texttt{8.2.0.2} and \texttt{3.6.0.1},
        \item \href{https://lmfdb.org/ModularCurve/Q/24.24.1.2/}{\texttt{24.24.1.2}}, the direct product of \texttt{8.2.0.2} and \texttt{3.12.0.1},
        \item \href{https://lmfdb.org/ModularCurve/Q/24.18.1.5/}{\texttt{24.18.1.5}}, the direct product of \texttt{8.6.0.1} and \texttt{3.3.0.1},
        \item \href{https://lmfdb.org/ModularCurve/Q/24.18.1.8/}{\texttt{24.18.1.8}}, the direct product of \texttt{8.6.0.6} and \texttt{3.3.0.1},
        \item \href{https://lmfdb.org/ModularCurve/Q/24.36.1.3/}{\texttt{24.36.1.3}}, the direct product of \texttt{8.6.0.6} and \texttt{3.6.0.1},
        \item \href{https://lmfdb.org/ModularCurve/Q/24.36.1.8/}{\texttt{24.36.1.8}}, the direct product of \texttt{8.6.0.1} and \texttt{3.6.0.1}.
    \end{enumerate}
\end{itemize}
\end{theorem}
At the moment, a full classification of curious Galois groups seems out of reach. One would likely need to investigate the Cummins--Pauli database \cite{cumminspauli} of subgroups of $\operatorname{SL}\left(2, \widehat{\ZZ}\right)$ of genus $0$ or $1$ and find a bounding condition to treat the case of non-trivial entanglements. See \cite{Entanglement1} , \cite{Entanglement2}, and \cite{Entanglement3} for more on non-trivial entanglements.

\begin{ack}
    The author would like to thank Harris Daniels, \'Alvaro Lozano-Robledo, Rakvi, Andrew Sutherland, and David Zureick-Brown for many helpful conversations during the writing of this paper.
\end{ack}

\section{Structure and methodology}\label{subsection structure and methodology}

In Section \ref{section background}, we briefly go over the objects of study of the paper; modular curves, elliptic curves and Galois representations. In Section \ref{section Galois images of elliptic curves defined over QQ}, we briefly introduce the work of Sutherland--Zywina and Rouse--Zureick-Brown in the classification of modular curves of prime-power level that contain infinitely many rational points. In Section \ref{section product groups}, we go over some group-theoretic lemmas. Let $p$ and $q$ be distinct primes and $M$ and $N$ be powers of $p$ and $q$ respectively. Let $\operatorname{H}_{p}$ be a group of level $M$ and let $\operatorname{H}_{q}$ be a group of level $N$. Then $\operatorname{H} = \operatorname{H}_{p} \times \operatorname{H}_{q}$ is a group of level $MN$. By an application of the Chinese remainder theorem, we can determine a subgroup of $\operatorname{GL}(2, \ZZ / MN \ZZ)$ that is isomorphic to $\operatorname{H}$. If $\operatorname{H}$ generates a modular curve with infinitely many rational points, then work by Sutherland and Zywina shows that $p$ and $q$ are primes less than or equal to $13$.

Section \ref{section main work} contains the proof of Theorem \ref{main Theorem} broken up into two theorems. We take direct products of pairs of groups from the Sutherland--Zywina database to get a group $\operatorname{H}$ of level $N$ where $N$ is a positive integer divisible by two primes. We focus on the groups $\operatorname{H}$ that are arithmetically admissible. We then use code from \cite{rouse_sutherland_zureick-brown_2022} to compute the genus. If the genus of $\operatorname{H}$ is strictly greater than $1$, then work by Faltings shows that the modular curve $\operatorname{X}_{\operatorname{H}}$ defined by $\operatorname{H}$ has finitely many points defined over $\QQ$. If the genus of $\operatorname{H}$ is equal to $0$, then either $\operatorname{X}_{\operatorname{H}}$ has no points defined over $\QQ$ or $\operatorname{X}_{\operatorname{H}}$ is isomorphic to $\PP^{1}$ and there is an elliptic curve $E/\QQ$ such that $\overline{\rho}_{E,N}(G_{\QQ})$ is conjugate to $\operatorname{H}$. In either case, $\operatorname{H}$ is not a curious Galois group. Hence, if $\operatorname{H}$ is a curious Galois group, then $\operatorname{X}_{\operatorname{H}}$ is an elliptic curve of positive rank.

There were a couple of ways to investigate groups $\operatorname{H}$ that generate modular curves of genus $1$ and positive rank to determine if they are or are not curious Galois groups. We first search the LMFDB for elliptic curves $E/\QQ$ such that $\overline{\rho}_{E,N}(G_{\QQ})$ is conjugate to $\operatorname{H}$. If there is such an example of an elliptic curve, then $\operatorname{H}$ is not a curious Galois group. If no such example of an elliptic curve exists in the LMFDB, then we inspect the subgroups of $\operatorname{H}$ of genus less than $2$. If $\operatorname{H}$ contains no subgroup of genus less than $2$, then finitely many rational points on $\operatorname{X}_{\operatorname{H}}$ correspond to elliptic curves $E/\QQ$ such that $\overline{\rho}_{E,N}(G_{\QQ})$. As $\operatorname{X}_{\operatorname{H}}$ contains infinitely many rational points, there is an elliptic curve $E'/\QQ$ such that $\overline{\rho}_{E,N}(G_{\QQ})$ is conjugate to $\operatorname{H}$ and hence, $\operatorname{H}$ is not a curious Galois group.

If all else fails, we suspect that $\operatorname{H}$ is a curious Galois group. Note that either $\operatorname{X}_{\operatorname{H}}(\QQ)$ is isomorphic to $\ZZ \times \ZZ / 2 \ZZ$ or $\ZZ \times \ZZ / 2 \ZZ \times \ZZ / 2 \ZZ$. A curious Galois group that is a direct product of groups contains some collection of proper subgroups $\operatorname{H}_{1}, \ldots, \operatorname{H}_{m}$ such that the modular curve $\operatorname{X}_{\operatorname{H}_{i}}$ generated by $\operatorname{H}_{i}$ is isomorphic to the modular curve $\operatorname{X}_{j}$ generated by $\operatorname{H}_{j}$ for $1 \leq i < j \leq m$. Moreover, $\operatorname{X}_{\operatorname{H}_{i}}$ is an elliptic curve such that $\operatorname{X}_{\operatorname{H}_{i}}(\QQ)$ is isomorphic to $\ZZ \times \ZZ / 2 \ZZ$ for all $1 \leq i \leq m$. Note that $[\operatorname{H} : \operatorname{H}_{i}] = p$ where $p = 2$ or $3$ and there is an isogeny $\phi \colon \operatorname{X}_{\operatorname{H}_{i}} \to \operatorname{X}_{\operatorname{H}}$ of degree $p$ defined over $\QQ$. For $1 \leq i \leq m$, let $\pi_{\operatorname{H}_{i}} \colon \operatorname{X}_{\operatorname{H}_{i}} \to \operatorname{X}_{\operatorname{H}}$ be the map from $\operatorname{X}_{\operatorname{H}_{i}}$ to the \textit{j}-line. Then there is a rational morphism $\psi \colon \operatorname{X}_{\operatorname{H}_{i}} \to \operatorname{X}_{\operatorname{H}}$ of degree $[\operatorname{H} : \operatorname{H}_{i}]$ such that the following diagram commutes.
\begin{center}
    \begin{tikzcd}
\operatorname{X}_{\operatorname{H}_{i}} \arrow[d, "\psi"'] \arrow[rd, "\pi_{\operatorname{H}_{i}}"] &                \\
\operatorname{X}_{\operatorname{H}} \arrow[r, "\pi_{\operatorname{H}}"']                            & \mathbb{P}^{1}
\end{tikzcd}
\end{center}
If for each rational point $P$ on $\operatorname{X}_{\operatorname{H}}$ there is a rational point $P_{i}$ on $\operatorname{X}_{\operatorname{H}_{i}}$ for some $1 \leq i \leq m$ such that $\pi_{\operatorname{H}}(P) = \pi_{\operatorname{H}_{i}}(P)$, then all rational points on $\operatorname{X}_{\operatorname{H}}$ correspond to rational points on $\operatorname{X}_{\operatorname{H}_{1}}, \ldots, \operatorname{X}_{\operatorname{H}_{m}}$ and $\operatorname{H}$ is a curious Galois group. For a point $P$ on $\operatorname{X}_{\operatorname{H}}$, let $\tau_{P} \colon \operatorname{X}_{\operatorname{H}} \to \operatorname{X}_{\operatorname{H}}$ be the translation-by-$P$ map and let $\mathcal{O}$ be the identity of $\operatorname{X}_{\operatorname{H}}$. Note that
$$\psi = \tau_{\mathcal{O}} \circ \psi = \tau_{\psi(\mathcal{O})} \circ \tau_{-\psi(\mathcal{O})} \circ \psi = \tau_{\psi(\mathcal{O})} \circ \left(\tau_{-\psi(\mathcal{O})} \circ \psi \right).$$
Moreover, note that $\tau_{-\psi(\mathcal{O})} \circ \psi$ is a rational morphism that maps $\mathcal{O}$ to $\mathcal{O}$ and thus, is an isogeny of degree $p$. Thus, $\tau_{-\psi(\mathcal{O})} \circ \psi = \phi$ or $-\phi$ and so, $\psi = \tau_{\psi(\mathcal{O})} \circ \phi'$ where $\phi' = \pm \phi$. We will show that as $\psi(\mathcal{O})$ ranges over the different types of elements of $\operatorname{X}_{\operatorname{H}}(\QQ)$, that there is a correspondence between $\operatorname{X}_{\operatorname{H}}(\QQ)$ and $\bigcup_{i=1}^{m}\operatorname{X}_{\operatorname{H}_{i}}(\QQ)$.

At times, $\operatorname{X}_{\operatorname{H}}(\QQ) \cong \ZZ \times \ZZ / 2 \ZZ$ and $\operatorname{H}$ contains a single subgroup $\operatorname{H}_{1}$ of index $3$. In this case, $\operatorname{X}_{\operatorname{H}_{1}}(\QQ) \cong \ZZ \times \ZZ / 2 \ZZ$ and there is an isogeny $\phi \colon \operatorname{X}_{\operatorname{H}_{1}} \to \operatorname{X}_{\operatorname{H}}$ of degree $3$ mapping $\operatorname{X}_{\operatorname{H}_{1}}(\QQ)$ bijectively onto $\operatorname{X}_{\operatorname{H}}(\QQ)$. Thus, for a rational point $P$ on $\operatorname{X}_{\operatorname{H}}$, we can find a point $P_{1}$ on $\operatorname{X}_{\operatorname{H}_{1}}$ such that $\phi(P_{1}) = P-\psi(\mathcal{O})$ or $-P+\psi(\mathcal{O})$ depending on whether $\psi = \tau_{\psi(\mathcal{O})} \circ \phi$ or $\tau_{\psi(\mathcal{O})} \circ -\phi$. Hence, $\psi(P_{1}) = P$ and the \textit{j}-invariant corresponding to $P_{1}$ is equal to the \textit{j}-invariant corresponding to $P$.

Other times, we will focus on subgroups $\operatorname{H}_{1}, \ldots, \operatorname{H}_{m}$ of $\operatorname{H}$ of index $2$ with $m = 2$ or $4$. In this case, $\operatorname{X}_{\operatorname{H}_{i}}(\QQ) = \left\langle g_{1,i}, g_{2,i} \right\rangle \cong \ZZ \times \ZZ / 2 \ZZ$ where $g_{1,i}$ is a generator of $\operatorname{X}_{\operatorname{H}_{i}}(\QQ)$ of infinite order and $g_{2,i}$ is a rational point on $\operatorname{X}$ of order $2$ for all $1 \leq i \leq m$ and $\operatorname{X}_{\operatorname{H}}(\QQ) \cong \ZZ \times \ZZ / 2 \ZZ$ or $\ZZ \times \ZZ / 2 \ZZ \times \ZZ / 2 \ZZ$. If $\operatorname{X}_{\operatorname{H}}(\QQ) = \left\langle g_{1}, g_{2} \right\rangle \cong \ZZ \times \ZZ / 2 \ZZ$, then the rational points on $\operatorname{X}_{\operatorname{H}}$ are of the form $[A] \cdot g_{1} + [B] \cdot g_{2}$ where $A$ is an integer and $B = 0$ or $1$. Let $g = [A] \cdot g_{1} + [B] \cdot g_{2}$ and $g' = [A'] \cdot g_{1} + [B'] \cdot g_{2}$ be rational points on $\operatorname{X}_{\operatorname{H}}$. We will say that the maps $\tau_{P}$ and $\tau_{P'}$ are of the same type if both $A \equiv A' \mod 2$ and $B \equiv B' \mod 2$. For $i \in 1, \ldots, m$, let $g'$ be an infinite generator on $\operatorname{X}_{\operatorname{H_{i}}}$. For all of our curious Galois groups that are direct products in this case, we have an isogeny $\phi \colon \operatorname{X}_{\operatorname{H}_{i}} \to \operatorname{X}_{\operatorname{H}}$ generated by the rational point on $\operatorname{X}_{\operatorname{H}_{i}}$ of order $2$, such that $\phi(g') = [\pm 2] \cdot g_{1} + D$ where $D$ is a $2$-torsion point on $\operatorname{X}_{\operatorname{H}}$.

It remains to show that for $1 \leq i,j \leq m$ with $i \neq j$, the translation-by-$\psi(\mathcal{O})$ maps are not of the same type. We find elliptic curves $E_{i}$ from the LMFDB database such that $\overline{\rho}_{E_{i},N}(G_{\QQ})$ is conjugate to a subgroup of $\operatorname{H}_{i}$ but not conjugate to a subgroup of $\operatorname{H}_{j}$ for $1 \leq i, j \leq m$ with $i \neq j$. Suppose then that $\psi_{i} \colon \operatorname{X}_{\operatorname{H}_{i}} \to \operatorname{X}_{\operatorname{H}}$ and $\psi_{j} \colon \operatorname{X}_{\operatorname{H}_{j}} \to \operatorname{X}_{\operatorname{H}}$ are the rational morphisms of degree $2$ that correspond with the maps to the \textit{j}-line and assume that $\psi_{i} = \tau_{\psi_{i}(\mathcal{O})} \circ \phi$ and $\psi_{j} = \tau_{\psi_{j}(\mathcal{O})} \circ \phi$ with $\psi_{i}$ and $\psi_{j}$ being of the same type. There is an integer $R$ such that $\psi_{i}([R] \cdot g')$ that corresponds to an elliptic curve $E/\QQ$ such that $\overline{\rho}_{E,N}(G_{\QQ})$ is conjugate to a subgroup of $\operatorname{H}_{i}$ but not to a subgroup of $\operatorname{H}_{j}$. As $\tau_{i}$ and $\tau_{j}$ are of the same type, there is an integer $S$ such that $\psi_{j}([S] \cdot g')$ that corresponds to an elliptic curve $E'/\QQ$ that is a quadratic twist of $E$ such that $\overline{\rho}_{E',N}(G_{\QQ})$ is conjugate to a subgroup of $\operatorname{H}_{j}$ but not to a subgroup of $\operatorname{H}_{i}$ but that is a contradiction. This is enough to show that $\operatorname{H}$ is a curious Galois group. On the other hand, there is one example of a curious Galois group $\operatorname{H}$ such that $\operatorname{X}_{\operatorname{H}}(\QQ) = \left\langle g_{1}, g_{2}, g_{3} \right\rangle \cong \ZZ \times \ZZ / 2 \ZZ \times \ZZ / 2 \ZZ$ with $g_{1}$ a generator of $\operatorname{X}_{\operatorname{H}}$ of infinite order. In this case, the isogeny $\phi \colon \operatorname{X}_{\operatorname{H}_{i}} \to \operatorname{X}_{\operatorname{H}}$ maps a generator $g'$ of $\operatorname{X}_{\operatorname{H}_{i}}$ of infinite order to $\pm g_{1} + D$ where $D$ is a $2$-torsion point of $\operatorname{X}_{\operatorname{H}}$. We end with an example of a curious Galois group which is not a direct product of Galois groups of prime-power level (greater than $1$).

\section{Background}\label{section background}

An elliptic curve defined over $\QQ$ is a smooth projective, algebraic curve of genus $1$ defined over $\QQ$ with a point defined over $\QQ$. Equivalently, an elliptic curve $E/\QQ$ is a smooth projective curve with the following homogeneous equation
$$E: Y^{2}Z + a_{1}XYZ + a_{3}YZ^{2} = X^{3} + a_{2}X^{2}Z + a_{4}XZ^{2} + a_{6}Z^{3}$$
with $a_{1}, a_{2}, a_{3}, a_{4},$ and $a_{6}$ being rational constants. With this equation, $E$ has the structure of an abelian group with identity being the single point at infinity, $\mathcal{O} = [0:1:0]$. After dehomogenizing with respect to $Z$, we get the following equation
$$E: y^{2} + a_{1}xy + a_{3}y = x^{3} + a_{2}x^{2} + a_{4}x + a_{6}.$$
Let $N$ be a positive integer. The set of all points on $E$ of order dividing $N$ with coordinates in $\overline{\QQ}$ is a group, denoted $E[N]$ and is isomorphic to $\ZZ / N \ZZ \times \ZZ / N \ZZ$. An element of $E[N]$ is called an $N$-torsion point. The group $G_{\QQ}:= \operatorname{Gal}\left(\overline{\QQ}/\QQ\right)$ acts on $E[N]$ for all positive integers $N$. From this action, we have the mod-$N$ Galois representation attached to $E$:
$$ \overline{\rho}_{E,N} \colon G_{\QQ} \to \operatorname{Aut}(E[N]).$$
After identifying $E[N] \cong \ZZ / N \ZZ \times \ZZ / N \ZZ$ and fixing a set of (two) generators of $E[N]$, we may consider the mod-$N$ Galois representation attached to $E$ as
$$\overline{\rho}_{E,N} \colon G_{\QQ} \to \operatorname{GL}(2,\ZZ / N \ZZ).$$

\subsection{Quadratic Twists}

Let $E : y^{2} = x^{3} + Ax + B$ be an elliptic curve defined over $\QQ$ and let $d$ be a non-zero integer. Then the quadratic twist of $E$ by $d$ is the elliptic curve $E^{(d)} : y^{2} = x^{3} + d^{2}Ax + d^{3}B$. Equivalently, $E^{(d)}$ is isomorphic to the elliptic curve $E^{(d)} : dy^{2} = x^{3} + Ax + B$. Then $E$ is isomorphic to $E^{(d)}$ over $\QQ(\sqrt{d})$ by the map
$$\phi \colon E \to E^{(d)}$$
defined by fixing $\mathcal{O}$ and mapping a non-zero point $(a,b)$ on $E$ to $\left(a,\frac{b}{\sqrt{d}}\right)$. Moreover, the \textit{j}-invariant of $E$ is equal to the \textit{j}-invariant of $E^{(d)}$. Conversely, if $E'/\QQ$ is an elliptic curve such that the \textit{j}-invariant of $E$ is equal to the \textit{j}-invariant of $E'$, with $\textit{j}_{E} \neq 0, 1728$, then $E$ is a (possibly trivial) quadratic twist of $E'$. We present an analogous definition of quadratic twists of subgroups of $\operatorname{GL}(2, \ZZ / N \ZZ)$.

\begin{definition}
Let $N$ be a positive integer and let $\operatorname{H_{1}}$ and $\operatorname{H_{2}}$ be subgroups of $\operatorname{GL}(2, \ZZ / N \ZZ)$. Suppose that $\operatorname{H_{2}}$ is the same as $\operatorname{H_{1}}$ up to multiplying some elements of $\operatorname{H_{1}}$ by $\operatorname{-Id}$ or that $\operatorname{H}_{2} = \left\langle \operatorname{H}_{1}, \operatorname{-Id} \right\rangle$. Then $\operatorname{H_{1}}$ is said to be a quadratic twist of $\operatorname{H_{2}}$ (and vice versa).
\end{definition}
Let $\operatorname{H_{1}}$ and $\operatorname{H_{2}}$ be subgroups of $\operatorname{GL}(2, \ZZ / N \ZZ)$. Let $\operatorname{H_{1}'} = \left\langle \operatorname{H_{1}}, \operatorname{-Id} \right\rangle$ and let $\operatorname{H_{2}'} = \left\langle \operatorname{H_{2}}, \operatorname{-Id} \right\rangle$. If $\operatorname{H_{1}'} = \operatorname{H_{2}'}$ then both $\operatorname{H}_{1}$ and $\operatorname{H}_{1}'$ are quadratic twists of both $\operatorname{H_{2}}$ and $\operatorname{H_{2}'}$. Conversely, suppose that $\operatorname{H_{1}}$ is a quadratic twist of $\operatorname{H_{2}}$. Then either the order of $\operatorname{H_{1}}$ is equal to the order of $\operatorname{H_{2}}$, the order of $\operatorname{H_{1}}$ is equal to twice the order of $\operatorname{H_{2}}$, or the order of $\operatorname{H_{2}}$ is equal to twice the order of $\operatorname{H_{1}}$. The second case happens when $\operatorname{H_{2}}$ does not contain $\operatorname{-Id}$ and $\operatorname{H_{1}} = \left\langle \operatorname{H_{2}}, \operatorname{-Id} \right\rangle$ and the third case happens when $\operatorname{H_{1}}$ does not contain $\operatorname{-Id}$ and $\operatorname{H_{2}} = \left\langle \operatorname{H_{1}}, \operatorname{-Id} \right\rangle$.

\begin{remark}
    Let $\operatorname{H}$ be a subgroup of $\operatorname{GL}(2, \ZZ / N \ZZ)$. Suppose that $\operatorname{H}$ does not contain $\operatorname{-Id}$ and that $\operatorname{H'} = \left\langle \operatorname{H}, \operatorname{-Id} \right\rangle$. Then $\operatorname{H}$ is a quadratic twist of $\operatorname{H'}$. Moreover, if a group $\operatorname{H_{1}}$ is a quadratic twist of a group $\operatorname{H_{2}}$, then either the order of $\operatorname{H_{1}}$ is equal to the order of $\operatorname{H_{2}}$ or the order of $\operatorname{H_{1}}$ (respectively, $\operatorname{H_{2}}$) is equal to twice the order of $\operatorname{H_{2}}$ (respectively, $\operatorname{H_{1}}$).
\end{remark}

\begin{remark}
Let $E/\QQ$ be an elliptic curve and let $N$ be an integer greater than or equal to $3$. Suppose that $\overline{\rho}_{E,N}(G_{\QQ}) = \operatorname{H}' = \left\langle \operatorname{H}, \operatorname{-Id}\right\rangle$ where $\operatorname{H}$ is a subgroup of $\operatorname{GL}(2, \ZZ / N \ZZ)$ that does not contain $\operatorname{-Id}$. Then there is a non-zero integer $d$ such that $\overline{\rho}_{E^{(d)},N}(G_{\QQ})$ is conjugate to $H$ (see Remark 1.1.3  and Section 10 in \cite{Rouse2021elladicIO}). Conversely, if $\overline{\rho}_{E,N}(G_{\QQ})$ is conjugate to $H$, then $\overline{\rho}_{E^{(d)},N}(G_{\QQ})$ is conjugate to $H'$ where $E^{(d)}$ is a quadratic twist of $E$ by a non-zero, square-free integer $d$, such that $\QQ(E[N])$ does not contain $\QQ(\sqrt{d})$.
\end{remark}

We can extend the definition of curious Galois groups to groups without $\operatorname{-Id}$. Let $\operatorname{H}$ be a subgroup of $\operatorname{GL}(2, \ZZ / N \ZZ)$ without $\operatorname{-Id}$. Then $\left\langle \operatorname{H}, \operatorname{-Id} \right\rangle$ is a curious Galois group if and only if $\operatorname{H}$ is a curious Galois group. Indeed, if $E/\QQ$ is an elliptic curve such that $\overline{\operatorname{\rho}}_{E,N}(\QQ)$ is conjugate to $\operatorname{H}$ (respectively, $\left\langle \operatorname{H}, \operatorname{-Id} \right\rangle$), then there is a non-zero integer $d$ such that $\overline{\rho}_{E^{(d)}, N}(G_{\QQ})$ is conjguate to $\left\langle \operatorname{H}, \operatorname{-Id} \right\rangle$ (respectively, $\operatorname{H}$). See Remark 1.1.3 and Section 10 in \cite{Rouse2021elladicIO} and particular, Lemma 5.24 and Corollary 5.25 in \cite{sutherland_2016}.

Let $\operatorname{K}$ be a number field. For each positive integer $N$, let $E[N]$ be the set of points on $E$ of order dividing $N$. Then $G_{\operatorname{K}} := \operatorname{Gal}(\overline{\operatorname{K}} / \operatorname{K})$ acts on $E[N]$. More specifically, let $\ell$ be a prime and let $T_{\ell}(E) := \underset{n \geq 0}{\varprojlim}E[\ell^{n}]$ be the Tate module of $E$. The action of $G_{\operatorname{K}}$ on $E[\ell^{n}]$ for each non-negative integer $n$ gives the mod-$\ell^{n}$ representation attached to $E$
$$\overline{\rho}_{E,\ell^{n}} \colon G_{\QQ} \to \operatorname{Aut}(E[\ell^{n}]).$$
After identifying $E\left[\ell^{n}
\right]$ with $\ZZ / \ell^{n} \ZZ \times \ZZ / \ell^{n} \ZZ$, we get that $\operatorname{Aut}(E[\ell^{n}]) \cong \operatorname{Aut}(\ZZ / \ell^{n} \ZZ \times \ZZ / \ell^{n} \ZZ) \cong \operatorname{GL}(2, \ZZ / \ell^{n} \ZZ)$.

In 1971, Serre \cite{serre1} proved the open image theorem.
\begin{theorem}[Serre, 1971, \cite{serre1}]
    Let $\operatorname{K}$ be a number field. Let $E/\operatorname{K}$ be an elliptic curve without complex multiplication. Then $\operatorname{\rho}_{E,\ell^{\infty}}(G_{\operatorname{K}})$ is an open subgroup of $T_{\ell}(\operatorname{K})$.
\end{theorem}
\begin{corollary}[Serre, 1971, \cite{serre1}]
    Let $\operatorname{K}$ be a number field. Let $E/\operatorname{K}$ be an elliptic curve without complex multiplication. Then there is a positive integer $C_{E,\operatorname{K}}$ (possibly dependent on both $K$ and $E$) such that if $p$ is a prime greater than $C_{E,\operatorname{K}}$, then $\overline{\rho}_{E,p}$ is surjective.
\end{corollary}

Serre went on to ask if the upper bound $C_{E,\operatorname{K}}$ is uniform among all elliptic curves $E$ defined over $\operatorname{K}$. In other words, is there a positive integer $C_{\operatorname{K}}$ (possibly dependent only on $\operatorname{K}$) such that if $E$ is a non-CM elliptic curve defined over $\operatorname{K}$, then $\overline{\rho}_{E,p}$ is surjective for all primes $p$ greater than $C_{\operatorname{K}}$? The consensus of the present day is yes when $\operatorname{K} = \QQ$ and that $C_{\QQ} = 37$. Much work has been done in the direction of answering this question and is treated on a case by case basis depending on the maximal subgroups of $\operatorname{GL}(2, \ZZ / p \ZZ)$. Let $E/\QQ$ be an elliptic curve and let $G := \overline{\rho}_{E,p}(G_{\QQ})$ for some prime $p$. Then there are five types of maximal subgroups of $\operatorname{GL}(2, \ZZ / p \ZZ)$.

\begin{enumerate}
    \item Maximal subgroups of $\operatorname{GL}(2, \ZZ / p \ZZ)$ that contain $\operatorname{SL}(2, \ZZ / p \ZZ)$.

By the fact that the determinant map
    $$\operatorname{Det} \colon \overline{\rho}_{E,p}(G_{\QQ}) \to \left(\ZZ / p \ZZ\right)^{\times}$$
    is surjective, if $G$ contains $\operatorname{SL}(2, \ZZ / p \ZZ)$, then $G = \operatorname{GL}(2, \ZZ / p \ZZ)$.

    \item Exceptional subgroups : maximal subgroups of $\operatorname{GL}(2, \ZZ / p \ZZ)$ whose reduction in $\operatorname{PGL}(E[p])$ is isomorphic to $\operatorname{A}_{4}$, $\operatorname{S}_{4}$, or $\operatorname{A}_{5}$.

    Serre himself proved that if $G$ is an exceptional subgroup, then there is a uniform upper bound $C$ such that if $p$ is a prime greater than $C$, then $G$ is not conjugate to an exceptional subgroup of $\operatorname{GL}(2, \ZZ / p \ZZ)$.

    \item Maximal subgroups that are Borel.

    Mazur in \cite{mazur1} proved that if $E/\QQ$ is a non-CM elliptic curve then $E$ does not have an isogeny of degree $p$ whenever $p$ is a prime greater than $37$.

    \item Maximal subgroups that are conjugate to the normalizer of the split Cartan group modulo $p$.

    Bilu--Parent in \cite{Bilu2008SerresUP} and Bilu--Parent--Rebolledo in \cite{Bilu2011RationalPO} proved that if $E/\QQ$ is a non-CM elliptic curve then there is a uniform upper bound $C$ such that if $p$ is a prime greater than $C$, then $G$ is not conjugate to the normalizer of the non-split Cartan group modulo $p$.

    \item Maximal subgroups that are conjugate to the normalizer of the non-split Cartan group modulo $p$.

    This is the final case that remains open. See \cite{Fourn2020ResidualGR} and \cite{Lemos2018SomeCO} for more work on this case.
\end{enumerate}

\subsection{Modular curves and elliptic curves}

Let $N$ be a positive integer and let $\operatorname{H}$ be a subgroup of $\operatorname{GL}(2, \ZZ / N \ZZ)$ such that the following three conditions are satisfied:

\begin{itemize}
    \item $\operatorname{-Id} \in \operatorname{H}$,

    \item $\operatorname{Det}(\operatorname{H}) = \left(\ZZ / N \ZZ\right)^{\times}$,

    \item $\operatorname{H}$ contains a matrix that is conjugate to $\begin{bmatrix}
        1 & 0 \\ 0 & -1
    \end{bmatrix}$ or $\begin{bmatrix}
        1 & 1 \\ 0 & -1
    \end{bmatrix}$.
\end{itemize}

Then there is a modular curve $\operatorname{X}_{\operatorname{H}}$ associated to $\operatorname{H}$. The third condition is necessary for this article because if $E/\QQ$ is an elliptic curve, then $\overline{\rho}_{E,N}(G_{\QQ})$ contains an element that behaves like complex conjugation (see Subsection 2.7 and Lemma 2.8 in \cite{SZ}).

Let $N$ be a positive integer. Let $\mathcal{F}_{N}$ denote the field of meromorphic functions of the Riemann surface $\operatorname{X}(N)$ whose $q$-expansions have coefficients in $K_{N} := \Q(\zeta_{N})$. For $f \in \mathcal{F}_{N}$ and $\gamma \in \operatorname{SL}(2, \Z)$, let $f \textbar_{\gamma} \in \mathcal{F}$ denote the modular function satisfying $f \textbar_{\gamma}(\tau) = f(\gamma \tau)$. For each $d \in (\Z / N \Z)^{\times}$, let $\sigma_{d}$ be the automorphism of $K_{N}$ such that $\sigma_{d}(\zeta_{N}) = \zeta_{N}^{d}$.

\begin{proposition}
The extension $\mathcal{F}_{N}$ of $\mathcal{F}_{1} = \Q(\textit{j})$ is Galois. There is a unique isomorphism
$$ \theta_{N} \colon \operatorname{GL}(2, \Z / N \Z) / \{\pm I\} \to \operatorname{Gal}(\mathcal{F}_{N} / \Q(\textit{j})) $$
such that the following hold for all $f \in \mathcal{F}_{N}$:

\begin{enumerate}

    \item for $g \in \operatorname{SL}(2, \Z / N \Z)$, we have $\theta_{N}(g)f = f \textbar_{\gamma^{t}}$ where $\gamma$ is any matrix in $\operatorname{SL}(2, \Z)$ that is congruent to $g$ modulo $N$.
    
    \item For $g = \left(\begin{array}{cc}
        1 & 0 \\
        0 & d
    \end{array}\right) \in \operatorname{GL}(2, \Z / N \Z)$, we have $\theta_{N}(g)f = \sigma_{d}(f)$.
\end{enumerate}
\end{proposition}
A subgroup $G$ of $\operatorname{GL}(2, \Z / N \Z)$ containing $\operatorname{-Id}$ and with full determinant mod-$N$ acts on $\mathcal{F}_{N}$ by $g \cdot f = \theta_{N}(g)f$ for $g \in G$ and $f \in \mathcal{F}_{N}$. Let $\mathcal{F}_{N}^{G}$ denote the subfield of $\mathcal{F}_{N}$ fixed by the action of $G$. The modular curve $\operatorname{X}_{G}$ associated to $G$ is the smooth projective curve with function field $\mathcal{F}^{G}_{N}$. The inclusion of fields $\mathcal{F}_{1} = \Q(\textit{j}) \subseteq \mathcal{F}_{N}$ induces a non-constant morphism
$$ \pi_{G} \colon \operatorname{X}_{G} \longrightarrow \operatorname{Spec}\Q[\textit{j}] \cup \{\infty\} = \PP^{1}_{\Q} $$
of degree $[\operatorname{GL}(2, \Z / N \Z) : G]$. If there is an inclusion of groups, $G \subseteq G' \subseteq \operatorname{GL}(2, \Z / N \Z)$, then there is an inclusion of fields $\Q(\textit{j}) \subseteq \mathcal{F}_{N}^{G'} \subseteq \mathcal{F}_{N}^{G}$ which induces a non-constant morphism $\psi \colon X_{G} \longrightarrow X_{G'}$ of degree $[G' \colon G]$ such that the following diagram commutes.
$$\begin{tikzcd}
\operatorname{X}_{G} \arrow[dd, "\psi"'] \arrow[rrdd, "\pi_{G}"] &  &                \\
                                               &  &                \\
\operatorname{X}_{G'} \arrow[rr, "\pi_{G'}"']                 &  & \mathbb{P}^{1}
\end{tikzcd}$$
For an elliptic curve $E/ \Q$ with $\textit{j}(E) \neq 0, 1728$ the group $\overline{\rho}_{E,N}(G_{\Q})$ is conjugate in $\operatorname{GL}(2,\Z / N \Z)$ to a subgroup of $G$ if and only if $\textit{j}(E)$ is an element of
$\pi_{G}(\operatorname{X}_{G}(\Q))$.

\begin{lemma}\label{double cover}
Let $\operatorname{H} \subseteq \operatorname{H}'$ be arithmetically admissible Galois groups such that the modular curves $\operatorname{X}_{\operatorname{H}'}$ and $\operatorname{X}_{\operatorname{H}}$ generated by $\operatorname{H}'$ and $\operatorname{H}$ respectively are elliptic curves defined over $\QQ$. Suppose that there is a an isogeny
$$\phi \colon \operatorname{X}_{\operatorname{H}} \to \operatorname{X}_{\operatorname{H}'}$$
of prime degree. Let $\tau$ be the map that translates all elements of $\operatorname{X}_{\operatorname{H}'}$ by $\psi(\mathcal{O})$. Then either $\psi = \tau \circ \phi$ or $\psi = \tau \circ -\phi$.

\end{lemma}

\begin{proof}

Let $\pi_{\operatorname{H}'}$ and $\pi_{\operatorname{H}}$ be the maps to the \textit{j}-line of $\operatorname{X}_{\operatorname{H}'}$ and $\operatorname{X}_{\operatorname{H}}$ respectively. Then there is a rational morphism $\psi \colon \operatorname{X}_{\operatorname{H}} \to \operatorname{X}_{\operatorname{H}'}$ of degree $2$ such that the following diagram commutes
$$\begin{tikzcd}
\operatorname{X}_{\operatorname{H}} \arrow[dd, "\psi"'] \arrow[rrdd, "\pi_{\operatorname{H}'}"] &  &                \\
                                               &  &                \\
\operatorname{X}_{\operatorname{H}'} \arrow[rr, "\pi_{\operatorname{H}'}"']                 &  & \mathbb{P}^{1}
\end{tikzcd}$$
Note that $\tau^{-1}$ is the map that translates all elements of $\operatorname{X}_{\operatorname{H}'}$ by $-\psi(\mathcal{O})$. Then $\tau^{-1} \circ \psi$ is a rational morphism mapping the identity of $\operatorname{X}_{\operatorname{H}}$ to the identity of $\operatorname{X}_{\operatorname{H}'}$. Hence, $\tau^{-1} \circ \psi$ is an isogeny of prime degree. By Proposition 4.12 in Chapter III, Section 4, of \cite{Silverman}, $\tau^{-1} \circ \psi = [M] \circ \phi$ where $[M]$ is the multiplication-by-$M$ map for some integer $M$. As the degree of $\tau^{-1} \circ \psi$ is prime, we have that $M = 1$ or $M = -1$. We conclude that $\psi = \tau \circ \phi$ or $\psi = \tau \circ -\phi$.

\end{proof}
\begin{theorem}[Faltings, 1983]\label{Faltings}
    Let $\mathcal{C}$ be an algebraic curve defined over $\QQ$. Let $g$ denote the genus of $\mathcal{C}$.
    \begin{itemize}
        \item If $g = 0$, then either $\mathcal{C}$ contains no points defined over $\QQ$ or $\mathcal{C}$ contains infinitely many points defined over $\QQ$. In the latter case, $\mathcal{C}$ is isomorphic to $\PP^{1}$.

        \item If $g = 1$, then either $\mathcal{C}$ contains no points defined over $\QQ$ or $\mathcal{C}$ contains points defined over $\QQ$. In the latter case, $\mathcal{C}$ is an elliptic curve and the set of points on $\mathcal{C}$ defined over $\QQ$ form a finitely-generated abelian group.

        \item If $g \geq 2$, then $\mathcal{C}$ has finitely many points defined over $\QQ$.
    \end{itemize}
\end{theorem}

\begin{thm}[Mazur \cite{mazur1}]\label{thm-mazur}
Let $E/\QQ$ be an elliptic curve. Then
\[
E(\Q)_\tor\simeq
\begin{cases}
\Z/M\Z &\text{with}\ 1\leq M\leq 10\ \text{or}\ M=12,\ \text{or}\\
\Z/2\Z\oplus \Z/2M\Z &\text{with}\ 1\leq M\leq 4.
\end{cases}
\]
\end{thm}

\subsection{Hilbert's Irreducibility Theorem} \label{HIT}

\begin{theorem}[Hilbert]

Let $f_{1}(X_{1}, \ldots, X_{r}, Y_{1}, \ldots, Y_{s}), \ldots, f_{n}(X_{1}, \ldots, X_{r}, Y_{1}, \ldots, Y_{s})$ be irreducible polynomials in the ring $\Q(X_{1}, \ldots, X_{r})[Y_{1}, \ldots, Y_{s}]$. Then there exists an $r$-tuple of rational numbers $(a_{1}, \ldots, a_{r})$ such that $f_{1}(a_{1}, \ldots, a_{r}, Y_{1}, \ldots, Y_{s}), \ldots, f_{n}(a_{1}, \ldots, a_{r}, Y_{1}, \ldots, Y_{s})$ are irreducible in the ring $\Q[Y_{1}, \ldots, Y_{s}]$.

\end{theorem}

In Section $7$ of \cite{SZ} there is the following statement:

Let $G$ be a subgroup of $\operatorname{GL}\left(2, \widehat{\ZZ}\right)$ of prime-power level such that the modular curve, $\operatorname{X}_{G}$ generated by $G$ is a curve of genus $0$. Define the set
$$\mathcal{S}_{G} := \bigcup_{G'} \pi_{G',G}(X_{G'}(\Q))$$
where $G'$ varies over the proper subgroups of $G$ that are conjugate to one of the groups that define a modular curve of prime-power level with infinitely many non-cuspidal, $\Q$-rational points from \cite{SZ} and $\pi_{G',G} \colon \operatorname{X}_{G'} \to \operatorname{X}_{G}$ is the natural morphism induced by the inclusion $G' \subseteq G$.

Then $\operatorname{X}_{G} \cong \PP^{1}$ and $\mathcal{S}_{G}$ is a \textit{thin} subset (see Chapter 3 of \cite{MR2363329} for the definition and properties of \textit{thin} sets) of $\operatorname{X}_{G}(\Q)$ in the language of Serre. The field $\Q$ is Hilbertian and in particular, $\PP^{1}(\Q) \cong \operatorname{X}_{G}(\Q)$ is not thin; this implies that $\operatorname{X}_{G}(\Q) \setminus \mathcal{S}_{G}$ cannot be thin and is infinite.

\begin{proposition} \label{HIT proposition}
Let $N$ be the power of a prime number. Let $\operatorname{H}$ be a subgroup of $\operatorname{GL}(2, \ZZ / N \ZZ)$ such that
\begin{enumerate}
    \item $\operatorname{H}$ contains $\operatorname{-Id}$,
    \item $\operatorname{H}$ contains a representative of complex conjugation mod-$N$,
    \item $\operatorname{H}$ has full determinant modulo $N$.
\end{enumerate}
If the modular curve $\operatorname{X}_{\operatorname{H}}$ defined by $\operatorname{H}$ is a genus $0$ curve that contains infinitely many non-cuspidal, $\QQ$-rational points, then there are infinitely many \textit{j}-invariants corresponding to elliptic curves over $\QQ$ such that the image of the mod-$N$ Galois representation is conjugate to $\operatorname{H}$ itself (not a proper subgroup of $\operatorname{H}$).
\end{proposition}

\begin{remark}\label{Hilbert's Irreducibility Theorem}
The proof of Proposition \ref{HIT proposition} follows directly from the discussion in Section $7$ of \cite{SZ}. In the case that $\operatorname{X}_{G}$ is genus $0$ and contains infinitely many non-cuspidal $\Q$-rational points, the set $\operatorname{X}_{G}(\QQ) \setminus \mathcal{S}_{G}$ is infinite. Hence, the set of non-cuspidal, $\Q$-rational points on $\operatorname{X}_{G}$ that are not on the modular curves defined by any proper subgroup of $G$ is infinite.
\end{remark}
The statement of Proposition \ref{HIT proposition} can be extended to all positive integers $N$ and to the following corollary.

\begin{corollary}\label{HIT corollary}
    Let $N$ be a positive integer and let $\operatorname{H}$ be an arithmetically admissible group of level $N$. Suppose that the genus of the modular curve $\operatorname{X}_{\operatorname{H}}$, generated by $\operatorname{H}$ is equal to $0$. Then either $\operatorname{X}_{\operatorname{H}}$ has no points defined over $\QQ$ or $\operatorname{X}_{\operatorname{H}}$ has a point defined over $\QQ$ that corresponds to an elliptic curve $E/\QQ$ such that $\overline{\rho}_{E,N}(G_{\QQ})$ is conjugate to $\operatorname{H}$.
\end{corollary}

\begin{corollary}\label{genus 1 corollary}
    Let $\operatorname{H}$ be a curious Galois group and let $\operatorname{X}_{\operatorname{H}}$ be the modular curve generated by $\operatorname{H}$. Then $\operatorname{X}_{\operatorname{H}}$ is an elliptic curve of positive rank.
\end{corollary}

For more information on modular curves, see Section 2 of \cite{SZ}. 

\section{Galois images of elliptic curves defined over $\QQ$}\label{section Galois images of elliptic curves defined over QQ}

In \cite{Rouse}, the authors classified the image of the $2$-adic Galois representation attached to non-CM elliptic curves defined over $\QQ$. 

\begin{theorem}[Rouse, Zureick-Brown, Theorem 1.1 in \cite{Rouse}]\label{RZB}

Let $\operatorname{H}$ be an open subgroup of $\operatorname{GL}(2, \ZZ_{2})$ and let $E/\QQ$ be an elliptic curve such that $\operatorname{\rho}_{E,2^{\infty}}(G_{\QQ})$ is conjugate to a subgroup of $\operatorname{H}$. Then one of the following holds:

\begin{itemize}
    \item The modular curve $\operatorname{X}_{\operatorname{H}}$ has infinitely many rational points.

    \item The curve $E$ has complex multiplication.

    \item The \textit{j}-invariant of $E$ is one of the following exception \textit{j}-invariants
    $$\left\{2^{11}, 2^{4} \cdot 17^{3}, \frac{4097^{3}}{2^{4}}, \frac{257^{3}}{2^{8}}, -\frac{857985^{3}}{62^{8}}, \frac{919425^{3}}{496^{4}}, -\frac{3 \cdot 18249920^{3}}{17^{16}}, -\frac{7 \cdot 1723187806080^{3}}{79^{16}}\right\}.$$
\end{itemize}
\end{theorem}

\begin{corollary}[Rouse, Zureick-Brown, Corollary 1.3 in \cite{Rouse}]
Let $E/\QQ$ be an elliptic curve without complex multiplication. Then the index of $\rho_{E,2^{\infty}}(G_{\QQ})$ in $\operatorname{GL}(2, \ZZ_{2})$ divides $64$ or $96$; and all such indices occur. Moreover, $\rho_{E,2^{\infty}}(G_{\QQ})$ is the inverse image of $\overline{\rho}_{E,32}(G_{\QQ})$. Moreover, there are $1208$ possible images of $\rho_{E,2^{\infty}}$. 
\end{corollary}
Similarly, in \cite{al-rCMGRs}, Lozano-Robledo classified the image of the $2$-adic Galois representation attached to elliptic curves defined over $\QQ$ with complex multiplication.

In \cite{SZ}, the authors classified the modular curves of prime-power level that contain infinitely many rational points. Their classification comes in the form of tables of groups in the appendix of \cite{SZ} and includes all of the groups from \cite{Rouse} that generate modular curves of $2$-power level with infinitely many points defined over $\QQ$. Note that the levels of the groups in \cite{SZ} are powers of primes less than or equal to $13$.

\section{Product groups}\label{section product groups}

Let $p_{1}, \ldots, p_{n}$ be distinct primes and let $M_{1}, \ldots,M_{n} \geq 2$ be integers such that $M_{i}$ is a power of $p_{i}$ for all $i \in \left\{1, \ldots, n\right\}$. Then we will call subgroups of $\operatorname{GL}(2, \ZZ / M_{1} \cdots M_{n} \ZZ) \cong \operatorname{GL}(2, \ZZ / M_{1} \ZZ) \times \ldots \times \operatorname{GL}(2, \ZZ / M_{n} \ZZ)$ that arise as direct products of proper subgroups of $\operatorname{GL}(2, \ZZ / M_{i} \ZZ)$ for $i \in \left\{1, \ldots, n\right\}$ as product groups.

\begin{lemma}\label{lemma CRT}
    Let $M, N \geq 2$ be coprime integers. Then $\ZZ / MN \ZZ \cong \ZZ / M \ZZ \times \ZZ / N \ZZ$
\end{lemma}

\begin{proof}
    Let $\phi_{MN} \colon \ZZ / MN \ZZ \to \ZZ / M \ZZ \times \ZZ / N \ZZ$ such that $\phi_{MN}(x) = (\phi_{M}(x), \phi_{N}(x))$ where $\phi_{M}(x)$ is the reduction of $x$ modulo $M$ and $\phi_{N}(x)$ is the reduction of $x$ modulo $N$. Then $\phi$ is an isomorphism.
\end{proof}

\begin{corollary} \label{corollary CRT}
    Let $M, N \geq 2$ be coprime integers. Then $\operatorname{GL}(2, \ZZ / MN \ZZ) \cong \operatorname{GL}(2, \ZZ / M \ZZ) \times \operatorname{GL}(2, \ZZ / N \ZZ)$.
\end{corollary}

\begin{proof}
    Let $\psi_{MN} \colon \operatorname{GL}(2, \ZZ / MN \ZZ) \to \operatorname{GL}(2, \ZZ / M \ZZ) \times \operatorname{GL}(2, \ZZ / N \ZZ)$ be the map that maps $A \in \operatorname{GL}(2, \ZZ / MN \ZZ)$ to $(\psi_{M}(A),\psi_{N}(A))$ where $\psi_{M}(A)$ is the reduction of the matrix $A \in \operatorname{GL}(2, \ZZ / MN \ZZ)$ modulo $M$ and $\psi_{N}(A)$ is the reduction of $A$ modulo $N$. Then $\psi_{MN}$ is a homomorphism. Let $A = \begin{bmatrix}
        a & b \\ c & d
    \end{bmatrix} \in \operatorname{Ker}(\psi_{MN})$. Then the reduction of $b$ and $c$ modulo $M$ and $N$ are both $0$. By Lemma \ref{lemma CRT}, both $b$ and $c$ are congruent to $0$ modulo $MN$. Moreover, the reduction of $a$ and $d$ modulo $M$ and $N$ are both $1$. By Lemma $\ref{lemma CRT}$, both $a$ and $d$ are congruent to $1$ modulo $MN$.

    Let $\begin{bmatrix}
        r_{1} & r_{2} \\ r_{3} & r_{4}
    \end{bmatrix} \in \operatorname{GL}(2, \ZZ / M \ZZ)$ and let $\begin{bmatrix}
        s_{1} & s_{2} \\ s_{3} & s_{4}
    \end{bmatrix} \in \operatorname{GL}(2, \ZZ / N \ZZ)$. We need to find a matrix $\begin{bmatrix}
        a & b \\ c & d
    \end{bmatrix} \in \operatorname{GL}(2, \ZZ / MN \ZZ)$ such that $\psi_{MN}\left(\begin{bmatrix}
        a & b \\ c & d
    \end{bmatrix}\right) = \left(\begin{bmatrix}
        r_{1} & r_{2} \\ r_{3} & r_{4}
    \end{bmatrix}, \begin{bmatrix}
        s_{1} & s_{2} \\ s_{3} & s_{4}
    \end{bmatrix}\right)$. In other words, we need to find elements $(a,b,c,d) \in \ZZ / MN \ZZ$ such that the reduction of $(a,b,c,d)$ modulo $M$ is equal to $(r_{1},r_{2},r_{3},r_{4})$ and the reduction of $(a,b,c,d)$ modulo $N$ is equal to $(s_{1},s_{2},s_{3},s_{4})$. This follows from Lemma \ref{lemma CRT}. Finally, the matrix $\begin{bmatrix}
        a & b \\ c & d
    \end{bmatrix}$ is invertible because $\psi_{M}\left(\begin{bmatrix}
        a & b \\ c & d
    \end{bmatrix}\right)$ and $\psi_{N}\left(\begin{bmatrix}
        a & b \\ c & d
    \end{bmatrix}\right)$ are both invertible.
\end{proof}

\begin{lemma}\label{lemma lifting}
    Let $M, N \geq 2$ be co-prime integers. Let $A \in \operatorname{GL}(2, \ZZ / M \ZZ)$. Then there is a matrix $\widehat{A} \in \operatorname{GL}(2, \ZZ / MN \ZZ)$ which reduces modulo $M$ to $A$.
\end{lemma}

\begin{proof}
    By Corollary \ref{corollary CRT}, $\operatorname{GL}(2, \ZZ / MN \ZZ)$ is isomorphic to $\operatorname{GL}(2, \ZZ / M \ZZ) \times \operatorname{GL}(2, \ZZ / N \ZZ)$. There is an obvious surjective map, namely the projective map $\pi_{MN,M} \colon \operatorname{GL}(2, \ZZ / M \ZZ) \times \operatorname{GL}(2, \ZZ N \ZZ) \to \operatorname{GL}(2, \ZZ / M \ZZ)$.
    Thus, there is an element $A' = (A,\operatorname{Id}) \in \operatorname{GL}(2, \ZZ / M \ZZ) \times \operatorname{GL}(2, \ZZ / N \ZZ)$ such that $\pi_{MN,N}(A') = A$. Next, we can identify $A' \in \operatorname{GL}(2, \ZZ / M \ZZ) \times \operatorname{GL}(2, \ZZ / N \ZZ)$ with a matrix $\widehat{A} \in \operatorname{GL}(2, \ZZ / MN \ZZ)$ with the desired properties. 
    
\end{proof}

\begin{remark}
    Let $M, N \geq 2$ be coprime integers. Then for a matrix $A \in \operatorname{GL}(2, \ZZ / M \ZZ)$, at times, we will need to fix a matrix in $\operatorname{GL}(2, \ZZ / MN \ZZ)$ that reduces modulo $M$ to $A$. We will denote such a matrix from $\operatorname{GL}(2, \ZZ / MN \ZZ)$ as $\widehat{A}$.
\end{remark}

\begin{lemma}\label{lifting with kernel}
    Let $M$ and $N$ be positive integers and let $\pi \colon \operatorname{GL}(2, \ZZ / MN \ZZ) \to \operatorname{GL}(2, \ZZ / M \ZZ)$ be the natural projection map. Let $\operatorname{H} = \left\langle g_{1}, \ldots, g_{m} \right\rangle$ be a subgroup of $\operatorname{GL}(2, \ZZ / M \ZZ)$. Then $\pi^{-1}(\operatorname{H}) = \left\langle \widehat{g_{1}}, \ldots, \widehat{g_{r}}, \operatorname{Ker}(\pi) \right\rangle$ where for $1 \leq i \leq r$, $\widehat{g_{i}}$ is any fixed element of $\operatorname{GL}(2, \ZZ / MN \ZZ)$ that reduces modulo $M$ to $g_{i}$.
\end{lemma}

\begin{proof}
    It is clear that $\pi^{-1}(\operatorname{H})$ is a group and it is clear that $\left\langle \widehat{g_{1}}, \ldots, \widehat{g_{r}}, \operatorname{Ker}(\pi) \right\rangle \subseteq \pi^{-1}(A)$. Let $x \in \pi^{-1}(\operatorname{H})$. Then $x = h_{1} \cdots h_{s}$ for some $h_{1}, \ldots, h_{s}$ where $h_{i} \in \left\{ g_{1}, \ldots, g_{r} \right\}$ for all $1 \leq i \leq s$. Let $\widehat{x} = \widehat{h_{1}} \cdots \widehat{h_{s}}$ where $\widehat{h_{i}} = \widehat{g_{i}}$ if and only if $h_{i} = g_{i}$ for all $1 \leq i \leq r$. it is clear that $\pi\left(\widehat{x}\right) = x$. Now let $\widehat{y} \in \operatorname{GL}(2, \ZZ / MN \ZZ)$ such that $\pi(\widehat{x}) = \pi(\widehat{y})$. Then
$$\operatorname{Id} = \pi(\widehat{y}) \cdot \pi(\widehat{x})^{-1} = \pi\left(\widehat{y} \cdot \widehat{x}^{-1}\right)$$
and hence, $\widehat{y} \cdot \widehat{x}^{-1} \in \operatorname{Ker}(\pi)$. Thus, we have $\widehat{y} \in \left\langle \widehat{x}, \operatorname{Ker} \right\rangle$. Hence, $\pi^{-1}(\operatorname{H}) \subseteq \left\langle \widehat{g_{1}}, \ldots, \widehat{g_{r}}, \operatorname{Ker}(\pi) \right\rangle$.
\end{proof}

\begin{lemma}\label{product groups}

Let $p$ and $q$ be distinct primes. Let $M \geq 2$ be a power of $p$ and let $N \geq 2$ be a power of $q$. Let $A$ be a subgroup of $\operatorname{GL}\left(2, \ZZ / M \ZZ\right)$ and let $B$ be a subgroup of $\operatorname{GL}\left(2, \ZZ / N \ZZ\right)$. Let $\pi_{1} \colon \operatorname{GL}\left(2, \ZZ / MN \ZZ\right) \to \operatorname{GL}(2, \ZZ / M \ZZ)$ and $\pi_{2} \colon \operatorname{GL}(2, \ZZ / MN \ZZ) \to \operatorname{GL}(\ZZ / N \ZZ)$ be the respective natural projection maps. Then the group of matrices in $\operatorname{GL}(2, \ZZ / MN\ZZ)$ that reduces modulo $M$ to $A$ and reduces modulo $N$ to $B$ is conjugate to $\left\langle \widehat{A}, \operatorname{Ker}(\pi_{1}) \right\rangle \bigcap \left\langle \widehat{B}, \operatorname{Ker}(\pi_{2}) \right\rangle$, where $\widehat{A}$ is any subgroup of $\operatorname{GL}(2, \ZZ / M N \ZZ)$ that reduces modulo $N$ to $A$ and $\widehat{B}$ is any subgroup of $\operatorname{GL}(2, \ZZ / MN \ZZ)$ that reduces modulo $N$ to $B$.

\end{lemma}

\begin{proof}
Let $H$ be a subgroup of $\operatorname{GL}\left(2, \ZZ / MN \ZZ\right)$ that reduces modulo $M$ to $A$. Using a similar proof to Lemma \ref{lifting with kernel}, we see that $H$ is conjugate to a subgroup of $\left\langle \widehat{A}, \operatorname{Ker}(\pi_{1}) \right\rangle$ where $\widehat{A}$ is any lift of $A$ to level $MN$. Also, $H$ reduces modulo $N$ to $B$. Similarly, we see that $H$ is conjugate to a subgroup of $\left\langle \widehat{B}, \operatorname{Ker}(\pi_{2}) \right\rangle$ where $\widehat{B}$ is any lift of $B$ to level $MN$. Hence, $H$ is conjugate to a subgroup of $\left\langle \widehat{A}, \operatorname{Ker}(\pi_{1}) \right\rangle \bigcap \left\langle \widehat{B}, \operatorname{Ker}(\pi_{2}) \right\rangle$. Conversely, a matrix in $\left\langle \widehat{A}, \operatorname{Ker}(\pi) \right\rangle \bigcap \left\langle \widehat{B}, \operatorname{Ker}(\pi_{2}) \right\rangle$ reduces modulo $M$ to a matrix in $A$ and simultaneously reduces modulo $N$ to a matrix in $B$.
\end{proof}

\begin{remark}
    The statements of the preceding lemmas and corollaries of this section can be extended to arbitrarily large factors.
\end{remark}

\begin{lemma}\label{noncurious subgroup critertion}
Let $N$ be a positive integer. Suppose that $\operatorname{H}$ is an arithmetically admissible group of level $N$. Let $\operatorname{X}_{\operatorname{H}}$ be the modular curve generated by $\operatorname{H}$ and suppose that $\operatorname{X}_{\operatorname{H}}(\QQ)$ is infinite. Suppose moreover that no proper arithmetically admissible subgroup of $\operatorname{H}$ generates a modular curve with infinitely many rational points. Then $\operatorname{H}$ is not a curious group of level $N$.
\end{lemma}

\begin{proof}
The group $\operatorname{H}$ has finitely many subgroups and we can apply the pigeonhole principle.
\end{proof}

\begin{corollary}\label{noncurious subgroup corollary}
    Let $N$ be a positive integer. Suppose that $\operatorname{H}$ is an arithmetically admissible group of level $N$. Let $\operatorname{X}_{\operatorname{H}}$ be the modular curve generated by $\operatorname{H}$ and suppose that $\operatorname{X}_{\operatorname{H}}(\QQ)$ is infinite. Suppose moreover that the genus of each proper arithmetically admissible subgroup of $\operatorname{H}$ is greater than or equal to $2$. Then $\operatorname{H}$ is not a curious group of level $N$.
\end{corollary}

\begin{proof}
     This is simply an application of Theorem \ref{Faltings} and Lemma \ref{noncurious subgroup critertion}.
\end{proof}

\begin{remark}
    In \cite{Rouse2021elladicIO}, the authors classified the image of the $\ell$-adic Galois representation attached to elliptic curves defined over $\QQ$ for all primes $\ell$. The associated code to \cite{Rouse2021elladicIO} gives a way to compute the genus of an open subgroup of $\operatorname{GL}\left(2, \widehat{\ZZ}\right)$. In this paper, we are interested in curious Galois groups that arise as direct products of proper subgroups of groups of the form $\operatorname{GL}(2,\ZZ_{p})$ for primes $p$. For example, let $p$ and $q$ are distinct primes and let $M,N \geq 2$ be powers of $p$ and $q$ respectively. Let $\operatorname{H}_{p}$ be a group of level $M$ and let $\operatorname{H}_{q}$ be a group of level $N$ and let $\operatorname{H} = \operatorname{H}_{p} \times \operatorname{H}_{q}$. We use Lemma \ref{product groups} to find a subgroup of $\operatorname{GL}(2, \ZZ / MN \ZZ)$ that is isomorphic to $\operatorname{H}$ and by abuse of notation, call this group $\operatorname{H}$ also. We then use the command \texttt{GL2Genus} to compute the genus of $\operatorname{H}$.
    
    Suppose that $\operatorname{H}'$ is an arithmetically admissible group of genus $1$ and level $N$. Let $\operatorname{X}_{\operatorname{H}'}$ be the modular curve generated by $\operatorname{H}'$. If $\operatorname{X}_{\operatorname{H}'}$ has finitely many points, then we conclude that $\operatorname{H}'$ is not a curious Galois group. Otherwise, suppose that $\operatorname{X}_{\operatorname{H}'}$ is an elliptic curve of positive rank. If $\operatorname{H}'$ has no subgroups that generate modular curves with infinitely many rational points, then we use Corollary \ref{noncurious subgroup corollary} to prove that $\operatorname{H}'$ is not curious. If otherwise, $\operatorname{H}'$ contains at least one proper subgroup $\operatorname{H}$ of genus $1$ and if there is no elliptic curve $E/\QQ$ in the LMFDB such that $\overline{\rho}_{E,N}(G_{\QQ})$ is conjugate to $\operatorname{H}'$ itself, then we suspect that $\operatorname{H}'$ is a curious Galois group. In this case, there might a proper arithmetically admissible subgroup $\operatorname{H}$ of $\operatorname{H}'$ of genus $1$ and index $3$ and $\operatorname{X}_{\operatorname{H}}(\QQ) = \operatorname{X}_{\operatorname{H}'}(\QQ)$ or $\operatorname{H}$ is an arithmetically admissible subgroup of $\operatorname{H}'$ of index $2$ and we use a collection of two or four modular curves that cover $\operatorname{X}_{\operatorname{H}'}$ by rational maps of degree $2$.
\end{remark}

\subsection{Notes about mod-$N$ images and adelic images}
Before we start the proof of Theorem \ref{main Theorem}, we pause to point out a difference between mod-$N$ images and adelic images. Consider the group $\operatorname{GL}\left(2, \widehat{\ZZ}\right)$ with label \href{https://beta.lmfdb.org/ModularCurve/Q/1.1.0.1/}{\texttt{1.1.0.1}}. By Theorem 1.2 of \cite{Greicius} and the subsequent work on the same page containing the theorem, there is no elliptic curve $E/\QQ$ such that $\rho_{E}(G_{\QQ}) = \operatorname{GL}\left(2, \widehat{\ZZ}\right)$. That is not to say that $\operatorname{GL}\left(2, \widehat{\ZZ}\right)$ is a curious Galois group by the definition found in this paper. Note that $\operatorname{GL}\left(2, \widehat{\ZZ}\right)$ is a group of level $1$ and for any elliptic curve $E/\QQ$, $\overline{\rho}_{E,1}(G_{\QQ})$ is conjugate to $\operatorname{GL}\left(2, \widehat{\ZZ}\right)$ modulo $1$. Classifying curious Galois groups by adelic image is a different question than the one investigated here (see \cite{rakvi2023classification}).

\section{Main work}\label{section main work}

Let $E/\QQ$ be an elliptic curve and let $N$ be a positive integer. Then $\rho_{E}(G_{\QQ})$ is a subset of $\underset{n \geq 0}{\varprojlim}\operatorname{GL}(2, \ZZ / N \ZZ) \cong \underset{\ell}{\Pi}\operatorname{GL}(2, \ZZ_{\ell})$. In this paper, we are interested in Galois images that are direct products. In other words, for primes $p_{1} < \ldots < p_{m}$, we discuss curious Galois groups of the form $\operatorname{H}_{p_{1}} \times \ldots \times \operatorname{H}_{p_{m}}$ where $\operatorname{H}_{p_{i}} \subseteq \operatorname{GL}(2, \ZZ_{p_{i}})$ for $1 \leq i \leq m$. Let us say that $\operatorname{H}_{p_{m}} = \operatorname{GL}(2, \ZZ / p_{m} \ZZ)$. Let $A = \operatorname{H}_{p_{1}} \times \ldots \times \operatorname{H}_{p_{m-1}}$ and let $B = A \times \operatorname{H}_{p_{m}}$. Let $\operatorname{X}_{A}$ be the modular curve generated by $A$ and let $\operatorname{X}_{B}$ be the modular curve generated by $B$. Then $\operatorname{X}_{A}$ is isomorphic to $\operatorname{X}_{B}$. For the rest of the paper, we will consider curious Galois groups of the form $\operatorname{H}_{p_{1}} \times \ldots \times \operatorname{H}_{p_{m}}$ with $\operatorname{H}_{p_{i}} \subsetneq \operatorname{GL}(2, \ZZ / p_{i})$ for all $1 \leq i \leq m$. We will call curious Galois groups of the form $\operatorname{H}_{p_{1}} \times \ldots \times \operatorname{H}_{p_{m}}$ with $\operatorname{H}_{p_{i}} \subsetneq \operatorname{GL}(2, \ZZ / p_{i})$ for all $1 \leq i \leq m$ to be curious Galois groups of type $(p_{1}, \ldots, p_{m})$.

Let $\operatorname{H}$ be a curious Galois group of type $(p_{1}, \ldots, p_{m})$. For $1 \leq i \leq m$, let $\operatorname{H}_{p_{i}}$ be the reduction of $\operatorname{H}$ to $p_{i}$-power level. Then $\operatorname{H}$ is a subgroup of the full preimage of $\operatorname{H}_{p_{i}}$ inside $\operatorname{GL}\left(2, \widehat{\ZZ}\right)$. Let $\operatorname{X}_{\operatorname{H}}$ be the modular curve defined by $\operatorname{H}$ and let $\operatorname{X}_{\operatorname{H}_{p_{i}}}$ be the modular curve defined by $\operatorname{H}_{p_{i}}$. Then there is a rational morphism
$$\psi \colon \operatorname{X}_{\operatorname{H}} \to \operatorname{X}_{p_{i}}$$
of (finite) degree $[\operatorname{H}_{p_{i}} : \operatorname{H}]$. As $\operatorname{X}$ contains infinitely many rational points, as $\psi$ maps rational points to rational points, and as the degree of $\psi$ is finite, we have that $\operatorname{X}_{p_{i}}$ has infinitely many rational points. Hence, $\operatorname{H}_{p_{i}}$ is a group from the SZ (\cite{SZ}) database. The rest of this article is devoted to classifying all curious Galois groups of type $(p_{1}, \ldots, p_{m})$ where $p_{1}, \ldots, p_{m}$ are distinct primes between $2$ and $13$. We will show that $m = 1$ or $2$ and fully classify curious Galois groups of type $(p_{1})$ and type $(p_{1},p_{2})$ as $p_{1}$ and $p_{2}$ range through all primes. First we classify the curious Galois groups of type $(p)$ where $p$ is a prime. This is a direct result of work from \cite{Rouse} and \cite{SZ}.

\begin{lemma}[Rouse, Zureick-Brown, \cite{Rouse}, Sutherland, Zywina, \cite{SZ}]

There are seven curious Galois groups of type $(p)$; namely the groups \href{https://beta.lmfdb.org/ModularCurve/Q/16.24.1.5/}{\texttt{16.24.1.5}}, \href{https://beta.lmfdb.org/ModularCurve/Q/16.24.1.10/}{\texttt{16.24.1.10}}, \href{https://beta.lmfdb.org/ModularCurve/Q/16.24.1.11/}{\texttt{16.24.1.11}}, \href{https://beta.lmfdb.org/ModularCurve/Q/16.24.1.13/}{\texttt{16.24.1.13}}, \href{https://beta.lmfdb.org/ModularCurve/Q/16.24.1.15/}{\texttt{16.24.1.15}}, \href{https://beta.lmfdb.org/ModularCurve/Q/16.24.1.17/}{\texttt{16.24.1.17}}, and \href{https://beta.lmfdb.org/ModularCurve/Q/16.24.1.19/}{\texttt{16.24.1.19}}.

\end{lemma}

\begin{proof}
Let $p$ be a prime number and let $\operatorname{H}$ be a proper, arithmetically admissible subgroup of $\operatorname{GL}\left(2, \widehat{\ZZ}\right)$ of $p$-power level. Suppose that $\operatorname{X}_{\operatorname{H}}$ has infinitely many non-cuspidal points defined over the rationals. Then $2 \leq p \leq 13$. If $p$ is an odd prime not equal to $11$, then the genus of $\operatorname{H}$ is equal to $0$ by \cite{SZ}. If $p = 2$ and $\operatorname{H}$ is not one of the seven groups in the hypothesis, then the genus of $\operatorname{H}$ is equal to $0$. Suppose that $p = 11$. Then by the work in \cite{SZ}, $\operatorname{H}$ is conjugate to the group \href{https://beta.lmfdb.org/ModularCurve/Q/11.55.1.1/}{\texttt{11.55.1.1}}. Let $E/\QQ$ be the elliptic curve \href{https://www.lmfdb.org/EllipticCurve/Q/232544/f/1}{\texttt{232544.f1}}. Then $\overline{\rho}_{E,11}(G_{\QQ})$ is conjugate to \href{https://beta.lmfdb.org/ModularCurve/Q/11.55.1.1/}{\texttt{11.55.1.1}}.
\end{proof}

We continue with classifying curious Galois groups of type $(p_{1},p_{2})$.

\begin{theorem} \label{odd theorem}
    Let $p_{1} < p_{2}$ be odd prime numbers and let $\operatorname{H}$ be a curious Galois group of type $(p_{1},p_{2})$. Then $\operatorname{H}$ is conjugate to the group \href{https://lmfdb.org/ModularCurve/Q/15.15.1.1/}{\texttt{15.15.1.1}}.
\end{theorem}

\begin{proof}

The proof is done on a case by case basis using the genera of modular curves that are the direct product of groups from \cite{SZ}. Given coprime integers $M, N \geq 2$, and a group $\operatorname{H}_{M}$ of level $M$ from \cite{SZ} and a group $\operatorname{H}_{N}$ of level $N$ from \cite{SZ}, we take their direct product $\operatorname{H}$ and find a subgroup of $\operatorname{GL}(2, \ZZ / MN \ZZ)$ that is isomorphic to $\operatorname{H}$ using Lemma \ref{product groups}. We use the associated code from \cite{Rouse2021elladicIO} to compute the genus of modular curves. Note that all groups from \cite{SZ} of $13$-power level are subgroups of \texttt{13.14.0.1}, all groups from \cite{SZ} of $11$-power level are subgroups of \texttt{11.55.1.1}, all groups of $7$-power level are subgroups of \texttt{7.8.0.1}, \texttt{7.21.0.1}, or \texttt{7.28.0.1}, all groups of $5$-power level are subgroups of \texttt{5.5.0.1}, \texttt{5.6.0.1}, or \texttt{5.10.0.1}, and all groups of $3$-power level are subgroups of \texttt{3.3.0.1}, \texttt{3.4.0.1}, or \texttt{9.27.0.1}.

\begin{enumerate}
    \item $(11,13)$
    
    The direct product of \texttt{11.55.1.1} and \texttt{13.14.0.1} is a group of genus $\geq 2$. By Theorem \ref{Faltings}, there are no curious Galois groups of type $(11,13)$.
    
    \item $(7,13)$
    
    The direct product of any of the groups \texttt{7.8.0.1}, \texttt{7.21.0.1}, or \texttt{7.28.0.1} and the group \texttt{13.14.0.1} is a group of genus $\geq 2$. By Theorem \ref{Faltings}, there are no curious Galois groups of type $(7,13)$.

    \item $(7,11)$

    The direct product of \texttt{11.55.1.1} and any of the groups \texttt{7.8.0.1}, \texttt{7.21.0.1}, or \texttt{7.28.0.1} is a group of genus $\geq 2$. By Theorem \ref{Faltings}, there are no curious Galois groups of type $(7,11)$.
    
    \item $(5,13)$

    The direct product of any of the groups \texttt{5.5.0.1}, \texttt{5.6.0.1}, or \texttt{5.10.0.1} with the group \texttt{13.14.0.1} is a group with genus $\geq 2$. By Theorem \ref{Faltings}, there are no curious Galois groups of type $(5,13)$.
    
    \item $(5,11)$

    The direct product of any of the groups \texttt{5.5.0.1}, \texttt{5.6.0.1}, or \texttt{5.10.0.1} with \texttt{11.55.1.1} is a group of genus $\geq 2$. By Theorem \ref{Faltings}, there are no curious Galois groups of type $(5,11)$.

    \item $(5,7)$

    The direct product of any of the groups \texttt{5.5.0.1}, \texttt{5.6.0.1}, or \texttt{5.10.0.1} with any of the groups \texttt{7.8.0.1}, \texttt{7.21.0.1}, or \texttt{7.28.0.1} is a group of genus $\geq 2$. By Theorem \ref{Faltings}, there are no curious Galois groups of type $(5,7)$.

    \item $(3,13)$

    The direct product of any of the groups \texttt{3.3.0.1}, \texttt{3.4.0.1}, or \texttt{9.27.0.1} with the group \texttt{13.14.0.1} is a group of genus $\geq 2$. By Theorem \ref{Faltings}, there are no curious Galois groups of type $(3,13)$.

    \item $(3,11)$

    The direct product of any of the groups \texttt{3.3.0.1}, \texttt{3.4.0.1}, or \texttt{9.27.0.1} with the group \texttt{11.55.1.1} is a group of genus $\geq 2$. By Theorem \ref{Faltings}, there are no curious Galois groups of type $(3,11)$.

    \item $(3,7)$

    Let $\operatorname{H}_{9}$ be a Galois group of level $9$ from \cite{SZ}. Then the direct product of $\operatorname{H}_{9}$ with any of the groups \texttt{7.8.0.1}, \texttt{7.21.0.1}, or \texttt{7.28.0.1} is a group of genus $\geq 2$. 

    Let $\operatorname{H}_{3}$ be a group from \cite{SZ} of level $3$ and let $\operatorname{H}_{7}$ be a group from \cite{SZ} of level $7$. Suppose that the direct product $\operatorname{H}$ of $\operatorname{H}_{3}$ and $\operatorname{H}_{7}$ is a group of genus $1$. Then $\operatorname{H}_{3}$ is either the group \texttt{3.4.0.1} and $\operatorname{H}_{7}$ is the group \texttt{7.8.0.1} and $\operatorname{H}$ is the group \texttt{21.32.1.1}. Note that in this case, \texttt{21.32.1.1} generates the modular curve $\operatorname{X}_{0}(21)$ which is an elliptic curve of rank $0$. Or on the other hand, $\operatorname{H}_{3}$ is the group \texttt{3.3.0.1} and $\operatorname{H}_{7}$ is the group \texttt{7.21.0.1} and $\operatorname{H}$ is the group \texttt{21.63.1.1}.

Suppose that $\operatorname{H}$ is the group \texttt{21.63.1.1}. Then $\operatorname{X}_{\operatorname{H}}$ is an elliptic curve isomorphic to the elliptic curve $E/\QQ$ with label \texttt{441.d2}. Moreover, $E(\QQ)$ is isomorphic to $\ZZ \times \ZZ / 3 \ZZ$. Note that all proper, arithmetically admissible subgroups of $\operatorname{H}$ have genus $\geq 2$. By Lemma \ref{noncurious subgroup critertion}, $\operatorname{H}$ is not a curious Galois group of type $(3,7)$. Thus, there are no curious Galois groups of type $(3,7)$.

\item $(3,5)$

Let $\operatorname{H}_{9}$ be a Galois group of level $9$ from \cite{SZ}. Then the direct product of $\operatorname{H}_{9}$ with the groups \texttt{5.5.0.1}, \texttt{5.6.0.1}, and \texttt{5.10.0.1}, is a group of genus $\geq 2$. Now let $\operatorname{H}_{25}$ be a Galois group from \cite{SZ} of level $25$. Then the genus of the direct product of $\operatorname{H}_{25}$ and the groups \texttt{3.3.0.1}, \texttt{3.4.0.1}, or \texttt{9.27.0.1} is greater than or equal to $2$.

Now let $\operatorname{H}_{3}$ be a group from \cite{SZ} of level $3$ and let $\operatorname{H}_{5}$ be a group from \cite{SZ} of level $5$ such that the genus of $\operatorname{H} = \operatorname{H}_{3} \times \operatorname{H}_{5}$ is equal to $1$. Moreover, suppose that the modular curve $\operatorname{X}_{\operatorname{H}}$ generated by $\operatorname{H}$ has infinitely many rational points. Then $\operatorname{H}_{3}$ is the group \texttt{3.3.0.1} and $\operatorname{H}_{5}$ is the group \texttt{5.5.0.1}, \texttt{5.10.0.1}, or \texttt{5.15.0.1}. Then $\operatorname{H}$ is \href{https://lmfdb.org/ModularCurve/Q/15.15.1.1/}{\texttt{15.15.1.1}}, \texttt{15.30.1.1}, or \texttt{15.45.1.1} respectively. The groups \href{https://lmfdb.org/ModularCurve/Q/15.15.1.1/}{\texttt{15.15.1.1}} and \texttt{15.45.1.1} correspond to the groups \texttt{[3Nn, 5S4]} and \texttt{[3Nn, 5Ns]} found near the beginning of Section 5 in \cite{Daniels2018SerresCO} where the authors concluded that the group \href{https://lmfdb.org/ModularCurve/Q/15.15.1.1/}{\texttt{15.15.1.1}} is a curious Galois group.

Finally, we have to prove that the groups \texttt{15.30.1.1} and \texttt{15.45.1.1} are not curious Galois groups. The groups \texttt{15.30.1.1} and \texttt{15.45.1.1} generate the same elliptic curve (namely, \texttt{225.c2}) of rank $1$. Finally, by Corollary \ref{noncurious subgroup corollary}, the groups \texttt{15.30.1.1} and \texttt{15.45.1.1} are not curious.
\end{enumerate}
\end{proof}

\begin{corollary}
    Let $m$ be an integer greater than or equal to $3$ and let $p_{1} < \ldots < p_{m}$ be odd primes. Then there are no curious Galois groups of type $(p_{1}, \ldots ,p_{m})$. 
\end{corollary}

\begin{proof}
    
    All that needs to be proven is the case when $m = 3$. Let $p_{1} < p_{2} < p_{3}$ be odd primes. Then $5 \leq p_{2} \leq p_{3}$. Assume by way of contradiction that $\operatorname{H}_{p_{i}}$ is a proper subgroup of $\operatorname{GL}(2, \ZZ_{p_{i}})$ for $1 \leq i \leq 3$ such that $\operatorname{H} = \operatorname{H}_{p_{1}} \times \operatorname{H}_{p_{2}} \times \operatorname{H}_{p_{3}}$ is a curious Galois group. By Theorem \ref{odd theorem} and its proof the group $\operatorname{H}_{p_{2}} \times \operatorname{H}_{p_{3}}$ generates a modular curve with finitely many rational points and thus, $\operatorname{H}$ generates a modular curve with finitely many rational points, a contradiction.
    
\end{proof}
Before we move on to the next theorem, we make a couple of definitions.

\begin{definition}
    Let $E/\QQ$ be an elliptic curve such that $E(\QQ)_{\texttt{tors}} \cong \ZZ \times \ZZ / 2 \ZZ$ (respectively, $\ZZ \times \ZZ / 2 \ZZ \times\ZZ / 2 \ZZ$) with $g_{1}$ an infinite generator of $E(\QQ)$ and $g_{2}$ a rational point on $E$ of order $2$ (respectively, $g_{3}$ another rational point of order $2$). For a point $P \in E$, let $\tau_{P} \colon E \to E$ be the translation-by-$P$ map on $E$. We say that for points $P = [A] \cdot g_{1} + [B] \cdot g_{2}$ (respectively $[A] \cdot g_{1} + [B] \cdot g_{2} + [C] \cdot g_{3}$) and $P' = [A'] \cdot g_{1} + [B'] \cdot g_{2}$ (respectively $[A'] \cdot g_{1} + [B'] \cdot g_{2} + [C'] \cdot g_{3}$) on $E$, $\tau_{P}$ and $\tau_{P'}$ are translation maps of the same type if $A \equiv A' \mod 2$ and $B \equiv B' \mod 2$ (respectively, and $C \equiv C' \mod 2$).
\end{definition}

\begin{theorem}
    Let $q$ be an odd prime number and suppose that $\operatorname{H}$ is a curious Galois group of type $(2,q)$. Then $\operatorname{H}$ is conjugate to the direct product of \texttt{8.2.0.1} and \texttt{13.14.0.1}, the direct product of \texttt{8.2.0.2} and \texttt{9.12.0.1}, \href{https://lmfdb.org/ModularCurve/Q/40.12.1.5/}{\texttt{40.12.1.5}}, \href{https://lmfdb.org/ModularCurve/Q/40.20.1.2/}{\texttt{40.20.1.2}}, \href{https://lmfdb.org/ModularCurve/Q/40.36.1.2/}{\texttt{40.36.1.2}}, \href{https://lmfdb.org/ModularCurve/Q/40.36.1.4/}{\texttt{40.36.1.4}}, \href{https://lmfdb.org/ModularCurve/Q/40.36.1.5/}{\texttt{40.36.1.5}}, \href{https://lmfdb.org/ModularCurve/Q/24.6.1.2/}{\texttt{24.6.1.2}}, \href{https://lmfdb.org/ModularCurve/Q/24.12.1.3/}{\texttt{24.12.1.3}}, \href{https://lmfdb.org/ModularCurve/Q/24.24.1.2/}{\texttt{24.24.1.2}}, \href{https://lmfdb.org/ModularCurve/Q/24.18.1.5/}{\texttt{24.18.1.5}}, \href{https://lmfdb.org/ModularCurve/Q/24.18.1.8/}{\texttt{24.18.1.8}}, \href{https://lmfdb.org/ModularCurve/Q/24.36.1.3/}{\texttt{24.36.1.3}}, or \href{https://lmfdb.org/ModularCurve/Q/24.36.1.8/}{\texttt{24.36.1.8}}
\end{theorem}

\begin{proof}

The proof is treated on a case by case basis. Let $M \geq 2$ be a power of $2$ and let $N \geq 3$ be a power of an odd prime number $q$. Let $\operatorname{H}_{2}$ be a group from \cite{SZ} of level $M$ and let $\operatorname{H}_{q}$ be a group from \cite{SZ} of level $N$. Let $\operatorname{H} = \operatorname{H}_{2} \times \operatorname{H}_{q}$. We use Lemma \ref{product groups} to find a subgroup of $\operatorname{GL}(2, \ZZ / MN \ZZ)$ that is isomorphic to $\operatorname{H}$ and use the code associated to \cite{rouse_sutherland_zureick-brown_2022} to compute the genus of $\operatorname{H}$. If the genus of $\operatorname{H}$ is equal to $1$, then we investigate $\operatorname{H}$ further. Note that all groups from \cite{SZ} of $2$-power level are subgroups of the groups \texttt{2.2.0.1}, \texttt{2.3.0.1}, \texttt{4.2.0.1}, \texttt{4.4.0.1}, \texttt{8.2.0.1}, or \texttt{8.2.0.2}. Note moreover that each subgroup from \cite{SZ} of level $32$ is a subgroup of a group from \cite{SZ} of level $16$.

Table \ref{Curious Galois groups} is a list with the left column being a group $\operatorname{H}$ of interest. If there is an example of an elliptic curve $E/\QQ$ from the LMFDB database such that $\overline{\rho}_{E,N}(G_{\QQ})$ is conjugate to $\operatorname{H}$, then we write the label of the elliptic curve on the right column of Table \ref{Curious Galois groups}. If we cannot find a corresponding elliptic curve but we know the group is not curious, we write \texttt{NOT CURIOUS}. Otherwise, the group is curious and we write \texttt{CURIOUS}. 

    \begin{enumerate}
        \item $(2,13)$

        Let $\operatorname{H}_{2}$ be a group from \cite{SZ} of level $2$, $4$, or $16$. Then the direct product of $\operatorname{H}_{2}$ with the group \texttt{13.14.0.1} generates a modular curve with finitely many points.

        Now let $\operatorname{H}_{2}$ be a group from \cite{SZ} of level $8$ and let $\operatorname{H}_{13}$ be a group from \cite{SZ} of level $13$. Then $\operatorname{H} = \operatorname{H}_{2} \times \operatorname{H}_{13}$ is an arithmetically admissible group with genus equal to $1$ if and only if $\operatorname{H}_{2}$ is the group \texttt{8.2.0.1} or \texttt{8.2.0.2} and $\operatorname{H}_{13}$ is the group \texttt{13.14.0.1}, \texttt{13.28.0.1}, or \texttt{13.28.0.2}. Let $\operatorname{H}_{2}$ be the group \texttt{8.2.0.2} and let $\operatorname{H}_{13}$ be the group \texttt{13.14.0.1}. Then $\operatorname{H} = \operatorname{H}_{2} \times \operatorname{H}_{13}$ generates the elliptic curve \texttt{832.e2}, which has finitely many rational points. To see this, we note that the \textit{j}-invariant of the rational points on the modular curve generated by \texttt{8.2.0.2} are of the form $2t^{2}+1728$ for some rational number $t$ (see the fifth row of page 25 in \cite{SZ}) and the \textit{j}-invariant of the rational points on the modular curve generated by \texttt{13.14.0.1} are of the form $\frac{(s^{2}+5s+13)(s^{4}+7s^{3}+20s^{2}+19s+1)^{3}}{s}$ for some rational number $s$ (see Table 2 in \cite{SZ}). Equating and homogenizing, we get an expression
        $$\frac{2t^{2}+1728z^{2}}{z^{2}} = \frac{(s^{2}+5sz+13z^{2})(s^{4}+7s^{3}z+20s^{2}z^{2}+19sz^{3}+z^{4})^{3}}{sz^{13}}.$$
After a little bit of algebra, we see that we get a curve in $\PP^{2}$ defined by the equation
$$f := sz^{11}(2t^{2}+1728z^{2}) - (s^{2}+5sz+13z^{2})(s^{4}+7sz^{3}+20s^{2}z^{2}+19sz^{3}+z^{4})^{3}$$
Using Magma \cite{magma}, one can generate the elliptic curve $E$ defined by $f$ with a specified rational point as the identity of $E$ and use the command \texttt{CremonaReference(E)} to find its Cremona reference. We see that $E$ is the elliptic curve \texttt{832.e2}, which has finitely many rational points. Hence, we are left with the groups $\operatorname{H}_{2}$ being \texttt{8.2.0.1} and $\operatorname{H}_{13}$ being \texttt{13.14.0.1}, \texttt{13.28.0.1}, or \texttt{13.28.0.2}.

        Let $E$ be the elliptic curve \texttt{11094.g2}. Then $\overline{\rho}_{E,104}(G_{\QQ})$ is conjugate to the direct product of \texttt{8.2.0.1} and \texttt{13.28.0.1}. Let $E$ be the elliptic curve \texttt{11094.g1}. Then $\overline{\rho}_{E,104}(G_{\QQ})$ is conjugate to the direct product of \texttt{8.2.0.1} and \texttt{13.28.0.2}. It remains to prove that the direct product of \texttt{8.2.0.1} and \texttt{13.14.0.1} is a curious Galois group.

        There are four proper, arithmetically admissible subgroups of the direct product of \texttt{8.2.0.1} and \texttt{13.14.0.1} of genus equal to $1$:

\begin{center}
\begin{itemize}

\item $\operatorname{H}_{1} = \left\langle [33,0,0,1],[83,78,0,1],[61,65,0,1],[30,65,13,17],[79,39,57,40] \right\rangle$, which generates an elliptic curve $\operatorname{X}_{1}$,
\item $\operatorname{H}_{2} = \left\langle [51,13,1,2],[55,91,1,2],[61,65,0,1],[49,0,0,1] \right\rangle$, which generates an elliptic curve $\operatorname{X}_{2}$,
\item $\operatorname{H}_{3} = \left\langle [33,0,0,1],[55,91,1,2],[8,91,1,2]
\right\rangle$, which generates an elliptic curve $\operatorname{X}_{3}$,
\item $\operatorname{H}_{4} = \left\langle [51,13,1,2],[83,78,0,1],[8,91,1,2],[49,0,0,1] \right\rangle$, which generates an elliptic curve $\operatorname{X}_{4}$.
\end{itemize}
\end{center}
Note that $\operatorname{H}_{1}$ is the direct product of \texttt{8.2.0.1} and \texttt{13.28.0.1} and $\operatorname{H}_{2}$ is the direct product of \texttt{8.2.0.1} and \texttt{13.28.0.2}. All four of the aforementioned groups are subgroups of $\operatorname{H}$ of index $2$. Moreover, the reduction of $\operatorname{H}_{3}$ and $\operatorname{H}_{4}$ modulo $8$ is \texttt{8.2.0.1} and the reduction modulo $13$ is \texttt{13.14.0.1}.

        Let $\operatorname{H}_{2}$ be the group \texttt{8.2.0.1} and let $\operatorname{H}_{13}$ be the group \texttt{13.14.0.1}. Then the direct product $\operatorname{H}$ of $\operatorname{H}_{2}$ and $\operatorname{H}_{13}$ generates the elliptic curve $\operatorname{X}/\QQ$ with label \texttt{832.f2}. Now let $\operatorname{H}_{13}$ be the group with label \texttt{13.28.0.1} (or \texttt{13.28.0.2}). Then the direct product of $\operatorname{H}_{2}$ with $\operatorname{H}_{13}$ generates the elliptic curve $\operatorname{X}'/\QQ$ with label \texttt{832.f1}. Note that $\operatorname{X}(\QQ) \cong \operatorname{X}'(\QQ) \cong \ZZ \times \ZZ / 2 \ZZ$. Moreover, all four of the groups $\operatorname{H}_{1}$, $\operatorname{H}_{2}$, $\operatorname{H}_{3}$, $\operatorname{H}_{4}$ generate modular curves that are isomorphic to the elliptic curve \texttt{832.f1}.
        
        Let $g_{1} = (-2,8)$, the generator of $\operatorname{X}(\QQ)$ of infinite order and let $g_{1}' = (1,3)$, the generator of $\operatorname{X}'(\QQ)$ of infinite order. Let $g_{2} = (-4,0)$, the rational point on $\operatorname{X}$ of order $2$ and let $g_{2}' = (2,0)$, the rational point on $\operatorname{X}'$ of order $2$. We have an isogeny generated by $g_{2}'$
        $$\phi \colon \operatorname{X}' \to \operatorname{X}$$
        of degree $2$ such that $\phi(g_{1}') = -2 \cdot g_{1}$. Let $\pi_{\operatorname{H}}$ and $\pi_{\operatorname{H}'}$ be maps from $\operatorname{X}$ and $\operatorname{X}'$ to the \textit{j}-line. Then there is a rational morphism $\psi \colon \operatorname{X}' \to \operatorname{X}$ of degree $2$ such that $\pi_{\operatorname{H}'} = \psi \circ \pi_{\operatorname{H}}$. By Lemma \ref{double cover}, $\psi = \tau_{\psi(\mathcal{O})} \circ \phi$ or $\psi = \tau_{\psi(\mathcal{O})} \circ -\phi$. Note that as $\psi$ is a rational morphism defined over $\QQ$, $\psi(\mathcal{O})$ is a rational point on $\operatorname{X}$.
        
        There are four types of elements that $\psi(\mathcal{O}) = [A] \cdot g_{1} + [B] \cdot g_{2}$ can be : $A$ is even and $B = 0$, $A$ is even and $B = 1$, $A$ is odd and $B = 0$, or $A$ is odd and $B = 1$. Let $E/\QQ$ be the elliptic curve \texttt{20736.a2}. Then $\overline{\rho}_{E,104}(G_{\QQ})$ is conjugate to $\operatorname{H}_{3}$. Let $E/\QQ$ be the elliptic curve \texttt{20736.a1}. Then $\overline{\rho}_{E,104}(G_{\QQ})$ is conjugate to $\operatorname{H}_{4}$. Thus, we have examples of elliptic curves such that the image of the mod-$104$ Galois representation is conjugate to each of the groups $\operatorname{H}_{1}$, $\operatorname{H}_{2}$, $\operatorname{H}_{3}$, and $\operatorname{H}_{4}$. For $i,j \in \left\{1,2,3,4\right\}$, $\operatorname{H}_{i}$ is conjugate to $\operatorname{H}_{j}$ if and only if $i = j$. The rational morphism $\psi_{i} \colon \operatorname{X}_{i} \to \operatorname{X}$ is of the form $\tau_{D_{i}} \circ [M_{i}] \circ \phi$ where $[M_{i}]$ is the identity map or the inversion map and $D_{i} = [A_{i}] \cdot g_{1} + [B_{i}] \cdot g_{2}$ for some integer $A_{i}$ and $B_{i} = 0$ or $1$. As we have found elliptic curves $E_{i}$ such that $\overline{\rho}_{E,104}(G_{\QQ})$ is conjugate to a subgroup of $\operatorname{H}_{i}$ but not to a subgroup of $\operatorname{H}_{j}$ when $i \neq j$, for all $i,j \in \left\{1,2,3,4\right\}$ we have that $A_{i} \not \equiv A_{j} \bmod 2$ or $B_{i} \not \equiv B_{j} \bmod 2$ when $i \neq j$. Thus, $\operatorname{H}$ is a curious Galois group.

\item $(2,11)$

Let $\operatorname{H}_{2}$ be the group \texttt{2.2.0.1}, \texttt{2.3.0.1}, \texttt{4.2.0.1}, \texttt{4.4.0.1}, \texttt{8.2.0.1}, or \texttt{8.2.0.2}. Then the direct product of $\operatorname{H}_{2}$ with \texttt{11.55.1.1} is a group of genus $\geq 2$. By Theorem \ref{Faltings}, there are no curious Galois groups of type $(2,11)$.
        
        \item $(2,7)$

        Let $\operatorname{H}_{2}$ be a group from \cite{SZ} of $2$-power level (greater than $1$) and let $\operatorname{H}_{7}$ be a group from \cite{SZ} of level $7$. Suppose $\operatorname{H} = \operatorname{H}_{2} \times \operatorname{H}_{7}$ is an arithmetically admissible group of genus $0$ or $1$. Suppose moreover that $\operatorname{H}$ generates a modular curve with infinitely many points defined over $\QQ$. Then the level of $\operatorname{H}$ is equal to $14$, $28$, or $56$. More specifically, $\operatorname{H}$ is the group \texttt{14.16.0.1}, \texttt{28.16.0.1}, \texttt{56.16.0.11}, or \texttt{56.16.0.12}; none of which are curious due to genus considerations.

        \item $(2,5)$

        Let $\operatorname{H}_{2}$ be a group from \cite{SZ} of $2$-power level (greater than $1$) and let $\operatorname{H}_{5}$ be a group from \cite{SZ} of level $25$. Then the direct product of $\operatorname{H}_{2}$ and $\operatorname{H}_{5}$ is a group of genus $\geq 2$. Next, if $\operatorname{H}_{2}$ is a group from \cite{SZ} of level $16$ and $\operatorname{H}_{5}$ is a group from \cite{SZ} of $5$-power level, then the direct product of $\operatorname{H}_{2}$ and $\operatorname{H}_{5}$ is a group of genus $\geq 2$. We move on to the case that $\operatorname{H}_{2}$ is a group from \cite{SZ} of level $2$, $4$, or $8$ and $\operatorname{H}_{5}$ is a group from \cite{SZ} of level $5$. We look up all groups from the LMFDB that are direct products of groups of level $2$, $4$, or $8$, and $5$, of genus $1$ with positive rank or genus $0$.

There are five curious Galois groups of type $(2,5)$. We investigate them below.

    \begin{itemize}

        \item \href{https://lmfdb.org/ModularCurve/Q/40.12.1.5/}{\texttt{40.12.1.5}}, the direct product of \texttt{8.2.0.1} and \texttt{5.6.0.1}, which generates the elliptic curve $\operatorname{X}/\QQ$ with label \texttt{320.a.2}. Note that $\operatorname{X}(\QQ) \cong \ZZ \times \ZZ / 2 \ZZ$. The group \href{https://lmfdb.org/ModularCurve/Q/40.12.1.5/}{\texttt{40.12.1.5}} contains four arithmetically admissible subgroups of index $2$, each of which generates an elliptic curve $\operatorname{X}'/\QQ$ isomorphic to \texttt{320.a1}. Moreover, $\operatorname{X}'(\QQ) \cong \ZZ \times \ZZ / 2 \ZZ$. 
        
        Let $g_{1} = (-5,40)$ be a generator of $\operatorname{X}(\QQ)$ of infinite order and let $g'_{1} = (11,-20)$ be a generator of $\operatorname{X}'(\QQ)$ of infinite order. Let $g_{2} = (-15,0)$ be the rational point on $\operatorname{X}$ of order $2$ and let $g_{2}' = (7,0)$ be the rational point on $E'$ of order $2$. There is an isogeny
        $$\phi \colon \operatorname{X}' \to \operatorname{X}$$
    generated by the rational point on $\operatorname{X}'$ of order $2$ such that $\phi(g_{1}') = 2 \cdot g_{1}$. 
                
        \begin{enumerate}
        \item $\operatorname{H}_{1} = \texttt{40.24.1.17}$ is a subgroup of \href{https://lmfdb.org/ModularCurve/Q/40.12.1.5/}{\texttt{40.12.1.5}} of index $2$ which is the direct product of \texttt{8.2.0.1} and \texttt{5.12.0.1} and generates an elliptic curve $\operatorname{X}_{1}$.
        \item $\operatorname{H}_{2} = \texttt{40.24.1.22}$ is a subgroup of \href{https://lmfdb.org/ModularCurve/Q/40.12.1.5/}{\texttt{40.12.1.5}} of index $2$ which is the direct product of \texttt{8.2.0.1} and \texttt{5.12.0.2} and generates an elliptic curve $\operatorname{X}_{2}$.
        \item $\operatorname{H}_{3} = \texttt{40.24.1.132}$ is a subgroup of \href{https://lmfdb.org/ModularCurve/Q/40.12.1.5/}{\texttt{40.12.1.5}} of index $2$ which generates the elliptic curve $\operatorname{X}_{3}$.
        The reduction of \texttt{40.24.1.137} modulo $5$ is conjugate to \texttt{5.6.0.1} and the reduction of \texttt{40.24.1.132} modulo $8$ is conjugate to \texttt{8.2.0.1}
        \item $\operatorname{H}_{4} = \texttt{40.24.1.137}$ is a subgroup of \href{https://lmfdb.org/ModularCurve/Q/40.12.1.5/}{\texttt{40.12.1.5}} of index $2$  which generates an elliptic curve $\operatorname{X}_{4}$.
        The reduction of \texttt{40.24.1.137} modulo $5$ is conjugate to \texttt{5.6.0.1} and the reduction of \texttt{40.24.1.137} modulo $8$ is conjugate to \texttt{8.2.0.1}

\end{enumerate}

Let $r,s \in \left\{1,2,3,4\right\}$. Then $\operatorname{H}_{r}$ is conjugate to $\operatorname{H}_{s}$ if and only if $r = s$. Let $E/\QQ$ be the elliptic curve \texttt{450.c3}. Then $\overline{\rho}_{E,40}(G_{\QQ})$ is conjugate to $\operatorname{H}_{1}$. Let $E/\QQ$ be the elliptic curve \texttt{450.c1}. Then $\overline{\rho}_{E,40}(G_{\QQ})$ is conjugate to $\operatorname{H}_{2}$. Let $E/\QQ$ be the elliptic curve \texttt{6400.d2}. Then $\overline{\rho}_{E,40}(G_{\QQ})$ is conjugate to $\operatorname{H}_{3}$. Let $E/\QQ$ be the elliptic curve \texttt{6400.d1}. Then $\overline{\rho}_{E,40}(G_{\QQ})$ is conjugate to $\operatorname{H}_{4}$. 

The rational points on $\operatorname{X}$ are of the form $[A] \cdot g_{1} + [B] \cdot g_{2}$ where $A$ is an integer and $B = 0$ or $1$. We partition $S = \operatorname{X}(\QQ)$ into four subsets: $S_{1}$ the subset of $S$ where $A$ is even and $B = 0$, $S_{2}$ the subset of $S$ where $A$ is even and $B = 1$, $S_{3}$ the subset of $S$ where $A$ is odd and $B = 0$, and $S_{4}$ the subset of $S$ where $A$ is odd and $B = 1$. 

For $i \in \left\{1,2,3,4\right\}$, the rational morphism $\psi_{i} \colon \operatorname{X}_{i} \to \operatorname{X}$ is of the form $\tau_{D_{i}} \circ [M_{i}] \circ \phi$ where $[M_{i}]$ is the identity map or the inversion map and $D_{i} = [A_{i}] \cdot g_{1} + [B_{i}] \cdot g_{2}$ for some integer $A_{i}$ and $B_{i} = 0$ or $1$. As we have found elliptic curves $E_{i}$ such that $\overline{\rho}_{E,40}(G_{\QQ})$ is conjugate to a subgroup of $\operatorname{H}_{i}$ but not to a subgroup of $\operatorname{H}_{j}$ when $i \neq j$, for all $i,j \in \left\{1,2,3,4\right\}$ we have that $A_{i} \not \equiv A_{j} \bmod 2$ or $B_{i} \not \equiv B_{j} \bmod 2$ when $i \neq j$. Thus, \href{https://lmfdb.org/ModularCurve/Q/40.12.1.5/}{\texttt{40.12.1.5}} is a curious Galois group.

        \item \href{https://lmfdb.org/ModularCurve/Q/40.20.1.2/}{\texttt{40.20.1.2}}, the direct product of \texttt{8.2.0.2} and \texttt{5.10.0.1}, which generates the elliptic curve $\operatorname{X}/\QQ$ with label \texttt{1600.c4}. Note that $\operatorname{X}(\QQ) \cong \ZZ \times \ZZ / 2 \ZZ$. The group \href{https://lmfdb.org/ModularCurve/Q/40.20.1.2/}{\texttt{40.20.1.2}} contains four arithmetically admissible subgroups of index $2$, two of which generate an elliptic curve $\operatorname{X}'/\QQ$ isormophic to \texttt{1600.c3}. Note that $\operatorname{X}'(\QQ) \cong \ZZ \times \ZZ / 2 \ZZ$. We list the four subgroups of index $2$ below (but only make use of the two that generate elliptic curves).
        
        Let $g_{1} = (18,-125)$ be a generator of $\operatorname{X}(\QQ)$ of infinite order and let $g_{1}' = (-2,25)$ be a generator of $\operatorname{X}'(\QQ)$ of infinite order. Let $g_{2} = (-7,0)$ be the rational point on $\operatorname{X}$ of order $2$ and let $g_{2}' = (3,0)$ be the rational point on $\operatorname{X}'$ of order $2$. There is an isogeny
        $$\phi \colon \operatorname{X}' \to \operatorname{X}$$
        generated by $g_{2}'$ such that $\phi(g_{1}') = -g_{1}$. 

        \begin{enumerate}
        \item $\operatorname{H}_{1} = \texttt{40.40.1.30}$ is a subgroup of \href{https://lmfdb.org/ModularCurve/Q/40.20.1.2/}{\texttt{40.20.1.2}} of index $2$ which generates an elliptic curve $\operatorname{X}_{1}$.
        The reduction of \texttt{40.40.1.30} modulo $5$ is conjugate to \texttt{5.10.0.1} and the reduction modulo $8$ is conjugate to \texttt{8.2.0.2}
        \item $\operatorname{H}_{2} = \texttt{40.40.1.35}$ is a subgroup of subgroup of \href{https://lmfdb.org/ModularCurve/Q/40.20.1.2/}{\texttt{40.20.1.2}} of index $2$ which generates an elliptic curve $\operatorname{X}_{2}$.
        The reduction of \texttt{40.40.1.35} modulo $5$ is conjugate to \texttt{5.10.0.1} and the reduction modulo $8$ is conjugate to \texttt{8.2.0.2}
        \item \texttt{40.40.1.39} is a subgroup of \href{https://lmfdb.org/ModularCurve/Q/40.20.1.2/}{\texttt{40.20.1.2}} of index $2$. The modular curve generated by \texttt{40.40.1.39} does not have any points defined over $\QQ_{2}$.
        \item \texttt{40.40.1.42} is a subgroup of \href{https://lmfdb.org/ModularCurve/Q/40.20.1.2/}{\texttt{40.20.1.2}} of index $2$. The modular curve generated by \texttt{40.40.1.42} does not have any points defined over $\QQ_{2}$.
\end{enumerate}
Note that $\operatorname{H}_{1}$ is not conjugate to $\operatorname{H}_{2}$. Note that both \texttt{40.40.1.30} and \texttt{40.40.1.35} are not curious as they do not contain proper subgroups that are arithmetically admissible with genus less than $2$.

The rational points on $\operatorname{X}$ are of the form $[A] \cdot g_{1} + [B] \cdot g_{2}$ where $A$ is an integer and $B = 0$ or $1$. We partition $S = \operatorname{X}(\QQ)$ into two subsets: $S_{1}$ the subset of $S$ where $B = 0$ and $S_{2}$ the subset of $S$ where $B = 1$. Let $i = 1$ or $2$. The rational morphism $\psi_{i} \colon \operatorname{X}_{i} \to \operatorname{X}$ is of the form $\tau_{D_{i}} \circ [M_{i}] \circ \phi$ where $[M_{i}]$ is the identity map or the inversion map and $D_{i} = [A_{i}] \cdot g_{1} + [B_{i}] \cdot g_{2}$ for some integer $A_{i}$ and $B_{i} = 0$ or $1$. As there are elliptic curves $E_{i}$ such that $\overline{\rho}_{E_{i},40}(G_{\QQ})$ is conjugate to a subgroup of $\operatorname{H}_{i}$ but not to a subgroup of $\operatorname{H}_{j}$ when $i \neq j$, for all $i,j \in \left\{1,2\right\}$ we have that $B_{i} \not \equiv B_{j} \bmod 2$ when $i \neq j$. Thus, \href{https://lmfdb.org/ModularCurve/Q/40.20.1.2/}{\texttt{40.20.1.2}} is a curious Galois group.

        \item \href{https://lmfdb.org/ModularCurve/Q/40.36.1.2/}{\texttt{40.36.1.2}}, the direct product of \texttt{8.6.0.2} and \texttt{5.6.0.1}, which generates the elliptic curve $\operatorname{X}/\QQ$ with label \texttt{320.c2}. Note that $\operatorname{X}(\QQ) \cong \ZZ \times \ZZ / 2 \ZZ \times \ZZ / 2 \ZZ$. The group \href{https://lmfdb.org/ModularCurve/Q/40.36.1.2/}{\texttt{40.36.1.2}} contains four arithmetically admissible subgroups of index $2$, each of which generate an elliptic curve $\operatorname{X}'/\QQ$ isomorphic to \texttt{320.c3}. Note that $\operatorname{X}'(\QQ) \cong \ZZ \times \ZZ / 2 \ZZ$. We list the four subgroups of index $2$ below.

        Let $g_{1} = (-3,3)$ be a generator of $\operatorname{X}(\QQ)$ of infinite order and let $g_{1}' = (1,-1)$ be a generator of $\operatorname{X}'(\QQ)$ of infinite order. Let $g_{2} = (6,0)$ and let $g_{3} = (-4,0)$ be rational points on $\operatorname{X}$ of order $2$ and let $g_{2}' = (2,0)$ be the rational point on $\operatorname{X}'$ of order $2$. There is an isogeny
        $$\phi \colon \operatorname{X}' \to \operatorname{X}$$
        generated by $g_{2}'$ such that $\phi(g_{1}') = g_{1}$.

        The rational points on $\operatorname{X}$ are of the form $[A] \cdot g_{1} + [B] \cdot g_{2} + [C] \cdot g_{3}$ where $A$ is and integer and $B, C = 0$ or $1$. We partition $S = \operatorname{X}(\QQ)$ into four subsets: $S_{1}$ the subset of $S$ where $B = C = 0$, $S_{2}$ the subset of $S$ where $B = 1$ and $C = 0$, $S_{3}$ the subset of $S$ where $B = 0$ and $C = 1$, and $S_{4}$ the subset of $S$ where $B = C = 1$.

\begin{enumerate}
        \item $\operatorname{H}_{1} = \texttt{40.72.1.35}$ is a subgroup of \href{https://lmfdb.org/ModularCurve/Q/40.36.1.2/}{\texttt{40.36.1.2}} of index $2$, which generates the elliptic curve $\operatorname{X}_{1}$. The reduction of \texttt{40.72.1.35} modulo $8$ is conjugate to \texttt{8.6.0.2} and modulo $5$ is conjugate to \texttt{5.6.0.1}.
        \item $\operatorname{H}_{2} = \texttt{40.72.1.44}$ is a subgroup of \href{https://lmfdb.org/ModularCurve/Q/40.36.1.2/}{\texttt{40.36.1.2}} of index $2$ which generates the elliptic curve $\operatorname{X}_{2}$. The reduction of \texttt{40.72.1.44} modulo $8$ is conjugate to \texttt{8.6.0.2} and modulo $5$ is conjugate to \texttt{5.6.0.1}.
        \item $\operatorname{H}_{3} = \texttt{40.72.1.67}$ is a subgroup of \href{https://lmfdb.org/ModularCurve/Q/40.36.1.2/}{\texttt{40.36.1.2}} of index $2$ which generates the elliptic curve $\operatorname{X}_{3}$. The reduction of \texttt{40.72.1.67} modulo $8$ is conjugate to \texttt{8.6.0.2} and modulo $5$ is conjugate to \texttt{5.6.0.1}.
        \item $\operatorname{H}_{4} = \texttt{40.72.1.76}$ is a subgroup of \href{https://lmfdb.org/ModularCurve/Q/40.36.1.2/}{\texttt{40.36.1.2}} of index $2$, which generates the elliptic curve $\operatorname{X}_{4}$. The reduction of \texttt{40.72.1.76} modulo $8$ is conjugate to \texttt{8.6.0.2} and modulo $5$ is conjugate to \texttt{5.6.0.1}.
        
    \end{enumerate}
    For $r,s \in \left\{1,2,3,4\right\}$, $\operatorname{H}_{r}$ is conjugate to $\operatorname{H}_{s}$ if and only if $r = s$. Let $E_{1}/\QQ$ be the elliptic curve \texttt{1734.d3}. Then $\overline{\rho}_{E_{1},40}(G_{\QQ})$ is conjugate to $\operatorname{H}_{1}$. Let $E_{2}/\QQ$ be the elliptic curve \texttt{1734.d1}. Then $\overline{\rho}_{E_{2},40}(G_{\QQ})$ is conjugate to $\operatorname{H}_{2}$. Let $E_{3}/\QQ$ be the elliptic curve \texttt{1734.d2}. Then $\overline{\rho}_{E_{3},40}(G_{\QQ})$ is conjugate to $\operatorname{H}_{3}$. Let $E_{4}/\QQ$ be the elliptic curve \texttt{1734.d4}. Then $\overline{\rho}_{E_{4},40}(G_{\QQ})$ is conjugate to $\operatorname{H}_{4}$.
    
    Let $i \in \left\{1,2,3,4\right\}$. The rational morphism $\psi_{i} \colon \operatorname{X}_{i} \to \operatorname{X}$ is of the form $\tau_{D_{i}} \circ [M_{i}] \circ \phi$ where $[M_{i}]$ is the identity map or the inversion map and $D_{i} = [A_{i}] \cdot g_{1} + [B_{i}] \cdot g_{2}$ for some integer $A_{i}$ and $B_{i} = 0$ or $1$. As we have found elliptic curves $E_{i}$ such that $\overline{\rho}_{E,40}(G_{\QQ})$ is conjugate to a subgroup of $\operatorname{H}_{i}$ but not to a subgroup of $\operatorname{H}_{j}$ when $i \neq j$, for all $i,j \in \left\{1,2,3,4\right\}$ we have that $A_{i} \not \equiv A_{j} \bmod 2$ or $B_{i} \not \equiv B_{j} \bmod 2$ when $i \neq j$. Thus, \href{https://lmfdb.org/ModularCurve/Q/40.36.1.2/}{\texttt{40.36.1.2}} is a curious Galois group.

        \item \href{https://lmfdb.org/ModularCurve/Q/40.36.1.4/}{\texttt{40.36.1.4}}, the direct product of \texttt{8.6.0.4} and \texttt{5.6.0.1}, which generates the elliptic curve $\operatorname{X}/\QQ$ with label \texttt{320.a4}. Note that $\operatorname{X}(\QQ) \cong \ZZ \times \ZZ / 2 \ZZ$. The group \href{https://lmfdb.org/ModularCurve/Q/40.36.1.4/}{\texttt{40.36.1.4}} contains four arithmetically admissible subgroups of index $2$, each of which generate the elliptic curve $\operatorname{X}'/\QQ$ isomorphic to \texttt{320.a3}. Note that $\operatorname{X}'(\QQ) \cong \ZZ \times \ZZ / 2 \ZZ$. We list the four subgroups of index $2$ below.

         Let $g = (3,-8)$ be a generator of $\operatorname{X}(\QQ)$ of infinite order and let $g' = (-2,1)$ be a generator of $\operatorname{X}'(\QQ)$ of infinite order. Let $g_{2} = (1,0)$ be the rational point on $\operatorname{X}$ of order $2$ and let $g_{2}' = (-1,0)$ be the rational point on $\operatorname{X}'$ of order $2$. There is an isogeny
        $$\phi \colon \operatorname{X}' \to \operatorname{X}$$
        generated by the rational point on $\operatorname{X}'$ of order $2$ such that $\phi(g') = 2 \cdot g$.

\begin{enumerate}
        \item $\operatorname{H}_{1} = \texttt{40.72.1.22}$ is a subgroup of \href{https://lmfdb.org/ModularCurve/Q/40.36.1.4/}{\texttt{40.36.1.4}} of index $2$, which is the direct product of \texttt{8.6.0.4} and \texttt{5.12.0.1} and generates an elliptic curve $\operatorname{X}_{1}$,
        \item $\operatorname{H}_{2} = \texttt{40.72.1.29}$ is a subgroup of \href{https://lmfdb.org/ModularCurve/Q/40.36.1.4/}{\texttt{40.36.1.4}} of index $2$, which is the direct product of \texttt{8.6.0.4} and \texttt{5.12.0.2} and generates an elliptic curve $\operatorname{X}_{2}$,
        \item $\operatorname{H}_{3} = \texttt{40.72.1.82}$ is a subgroup of \href{https://lmfdb.org/ModularCurve/Q/40.36.1.4/}{\texttt{40.36.1.4}} of index $2$ which generates an elliptic curve $\operatorname{X}_{3}$. The reduction of \texttt{40.72.1.82} modulo $8$ is conjugate to \texttt{8.6.0.4} and the reduction modulo $5$ is conjugate to \texttt{5.6.0.1},
        \item $\operatorname{H}_{4} = \texttt{40.72.1.89}$ is a subgroup of \href{https://lmfdb.org/ModularCurve/Q/40.36.1.4/}{\texttt{40.36.1.4}} of index $2$ which generates an elliptic curve $\operatorname{X}_{4}$. The reduction of \texttt{40.72.1.89} modulo $8$ is conjugate to \texttt{8.6.0.4} and the reduction modulo $5$ is conjugate to \texttt{5.6.0.1}.
    \end{enumerate}
    Let $r,s \in \left\{1,2,3,4\right\}$. Then $\operatorname{H}_{r}$ is conjugate to $\operatorname{H}_{s}$ if and only if $r = s$. Let $E_{1}/\QQ$ be the elliptic curve \texttt{198.c1}. Then $\overline{\rho}_{E_{1},40}(G_{\QQ})$ is conjugate to $\operatorname{H}_{1}$. Let $E_{2}/\QQ$ be the elliptic curve \texttt{198.c3}. Then $\overline{\rho}_{E_{2},40}(G_{\QQ})$ is conjugate to $\operatorname{H}_{2}$. Let $E_{3}/\QQ$ be the elliptic curve \texttt{768.a1}. Then $\overline{\rho}_{E_{3},40}(G_{\QQ})$ is conjugate to $\operatorname{H}_{3}$. Let $E_{4}/\QQ$ be the elliptic curve \texttt{768.a3}. Then $\overline{\rho}_{E_{4},40}(G_{\QQ})$ is conjugate to $\operatorname{H}_{4}$.
    
    The rational points on $\operatorname{X}$ are of the form $[A] \cdot g_{1} + [B] \cdot g_{2}$ where $A$ is an integer and $B = 0$ or $1$. We partition $S = \operatorname{X}(\QQ)$ into four subsets: $S_{1}$ the subset of $S$ where $A$ is even and $B = 0$, $S_{2}$ the subset of $S$ where $A$ is even and $B = 1$, $S_{3}$ the subset of $S$ where $A$ is odd and $B = 0$, and $S_{4}$ the subset of $S$ where $A$ is odd and $B = 1$. 
    
    The rational morphism $\psi_{i} \colon \operatorname{X}_{i} \to \operatorname{X}$ is of the form $\tau_{D_{i}} \circ [M_{i}] \circ \phi$ where $[M_{i}]$ is the identity map or the inversion map and $D_{i} = [A_{i}] \cdot g_{1} + [B_{i}] \cdot g_{2}$ for some integer $A_{i}$ and $B_{i} = 0$ or $1$. As we have found elliptic curves $E_{i}$ such that $\overline{\rho}_{E,40}(G_{\QQ})$ is conjugate to a subgroup of $\operatorname{H}_{i}$ but not to a subgroup of $\operatorname{H}_{j}$ when $i \neq j$, for all $i,j \in \left\{1,2,3,4\right\}$ we have that $A_{i} \not \equiv A_{j} \bmod 2$ or $B_{i} \not \equiv B_{j} \bmod 2$ when $i \neq j$. Thus, $\operatorname{H}$ is a curious Galois group. Thus, \href{https://lmfdb.org/ModularCurve/Q/40.36.1.4/}{\texttt{40.36.1.4}} is a curious Galois group.

        \item \href{https://lmfdb.org/ModularCurve/Q/40.36.1.5/}{\texttt{40.36.1.5}}, the direct product of \texttt{8.6.0.5} and \texttt{5.6.0.1}.

Note that \href{https://lmfdb.org/ModularCurve/Q/40.36.1.5/}{\texttt{40.36.1.5}} is a subgroup of \href{https://lmfdb.org/ModularCurve/Q/40.12.1.5/}{\texttt{40.12.1.5}} of index $3$. Let $\operatorname{X}'$ be the modular curve generated by \href{https://lmfdb.org/ModularCurve/Q/40.36.1.5/}{\texttt{40.36.1.5}}, let $\operatorname{X}$ be the modular curve generated by \href{https://lmfdb.org/ModularCurve/Q/40.12.1.5/}{\texttt{40.12.1.5}}, and let $\operatorname{X}_{1}$, $\operatorname{X}_{2}$, $\operatorname{X}_{3}$, and $\operatorname{X}_{4}$, be the modular curves generated by the four subgroups of \href{https://lmfdb.org/ModularCurve/Q/40.12.1.5/}{\texttt{40.12.1.5}} of index $2$ listed above. Let $P$ be a point on $\operatorname{X}'$ defined over $\QQ$. Then $P$ is a point on $\operatorname{X}$ defined over $\QQ$ and hence, $P$ is a point on $\operatorname{X}_{1}$, $\operatorname{X}_{2}$, $\operatorname{X}_{3}$, or $\operatorname{X}_{4}$ defined over $\QQ$. As $3$ is not even, this means that $P$ cannot correspond to an elliptic curve $E/\QQ$ such that $\overline{\rho}_{E,40}(G_{\QQ})$ is conjugate to \href{https://lmfdb.org/ModularCurve/Q/40.36.1.5/}{\texttt{40.36.1.5}} precisely. Hence, \href{https://lmfdb.org/ModularCurve/Q/40.36.1.5/}{\texttt{40.36.1.5}} is a curious Galois group.

    \end{itemize}
    
        \item $(2,3)$

        Let $\operatorname{H}_{2}$ be a group of $2$-power level and let $\operatorname{H}_{3}$ be a group of $3$-power level such that $\operatorname{H} = \operatorname{H}_{2} \times \operatorname{H}_{3}$ is a curious Galois group whose level is divisible by $6$. By genera computations using the code associated to \cite{rouse_sutherland_zureick-brown_2022}, the level of $\operatorname{H}_{2}$ is at most $16$ and the level of $\operatorname{H}_{3}$ is at most $9$. Moreover, $\operatorname{H}$ is not a group of level $16 \cdot 9$.

        We investigate the curious Galois groups that are direct products of groups of level $8$ and $9$ from \cite{SZ}. The candidates of level $8 \cdot 9$ are the groups that are the direct product of the following pairs of groups.

        \begin{itemize}
                    \item The direct product of \texttt{8.2.0.1} and \texttt{9.12.0.2}; genus $0$

                    By Corollary \ref{genus 1 corollary}, the direct product of \texttt{8.2.0.1} and \texttt{9.12.0.2} is not a curious Galois group.
            \item The direct product of \texttt{8.2.0.2} and \texttt{9.12.0.2}; genus $0$

            By Corollary \ref{genus 1 corollary}, the direct product of \texttt{8.2.0.2} and \texttt{9.12.0.2} is not a curious Galois group.

            \item The direct product of \texttt{8.2.0.1} and \texttt{9.12.0.1}; genus $1$ and rank $0$
            
            Note that the direct product of the groups \texttt{8.2.0.1} and \texttt{9.12.0.1} generates the elliptic curve \texttt{576.f4} which is an elliptic curve of rank $0$. Hence, the direct product of the groups \texttt{8.2.0.1} and \texttt{9.12.0.1} is not a curious Galois group.

            \item The direct product of \texttt{8.6.0.6} and \texttt{9.12.0.1}; genus $1$ and rank $1$

            Let $\operatorname{H}$ be the direct product of the groups \texttt{8.6.0.6} and \texttt{9.12.0.1}. Then $\operatorname{H}$ generates a modular curve isomorphic to \texttt{576.e4}. Note that \texttt{576.e4} has rank $1$. Note that there is no proper, arithmetically admissible subgroup of $\operatorname{H}$ of genus less than $2$. By a pigeonhole principle argument, $\operatorname{H}$ is not a curious Galois group.

            \item The direct product of \texttt{8.2.0.2} and \texttt{9.12.0.1}; genus $1$ and rank $1$

            Let $\operatorname{H}$ denote the direct product of \texttt{8.2.0.2} and \texttt{9.12.0.1} and let $\operatorname{H}'$ denote the direct product of \texttt{8.2.0.2} and \texttt{3.12.0.1} (in other words, the group \href{https://lmfdb.org/ModularCurve/Q/24.24.1.2/}{\texttt{24.24.1.2}}). Note that \texttt{9.12.0.1} is the full Borel group modulo $9$ and \texttt{3.12.0.1} is the full split Cartan group modulo $3$. Let $E/\QQ$ be an elliptic curve. Then $\overline{\rho}_{E,72}(G_{\QQ})$ is conjugate to $\operatorname{H}$ if and only if $E$ is $3$-isogenous to an elliptic curve $E'/\QQ$ such that $\overline{\rho}_{E',24}(G_{\QQ})$ is conjugate to \href{https://lmfdb.org/ModularCurve/Q/24.24.1.2/}{\texttt{24.24.1.2}}. In other words, $\operatorname{H}$ is a curious Galois group if and only if \href{https://lmfdb.org/ModularCurve/Q/24.24.1.2/}{\texttt{24.24.1.2}} is a curious Galois group. As \href{https://lmfdb.org/ModularCurve/Q/24.24.1.2/}{\texttt{24.24.1.2}} was proven to be a curious Galois group in \cite{Chiloyan20232adicGI}, the direct product of \texttt{8.2.0.2} and \texttt{9.12.0.1} is a curious Galois group. 

            \item The direct product of \texttt{8.6.0.1} and \texttt{9.12.0.1}; genus $1$ and rank $1$

            Let $\operatorname{H}$ be the direct product of \texttt{8.6.0.1} and \texttt{9.12.0.1}. Then $\operatorname{H}$ generates a modular curve that is isomorphic to the elliptic curve \texttt{576.e4} which is rank $1$. Note that there is no proper, arithmetically admissible subgroup of $\operatorname{H}$ of genus less than $2$. By a pigeonhole principle argument, $\operatorname{H}$ is not a curious Galois group.

            \item The direct product of \texttt{8.6.0.4} and \texttt{9.12.0.1}; genus $1$ and rank $0$

            Let $\operatorname{H}$ denote the direct product of \texttt{8.6.0.4} and \texttt{9.12.0.1}. Then $\operatorname{H}$ generates a modular curve that is isomorphic to \texttt{576.f3}. The rank of \texttt{576.f3} is equal to $0$. Hence, $\operatorname{H}$ is not a curious Galois group.

            \item The direct product of \texttt{8.6.0.5} and \texttt{9.12.0.1}; genus $1$ and rank $0$

            The group \texttt{8.6.0.5} is a subgroup of the group \texttt{8.2.0.1}. By the fact that the direct product of \texttt{8.2.0.1} and \texttt{9.12.0.1} generates a modular curve with finitely many points, so does the direct product of \texttt{8.6.0.5} and \texttt{9.12.0.1}. Hence, the direct product of \texttt{8.6.0.5} and \texttt{9.12.0.1} is not a curious Galois group.
       \end{itemize}

        The rest of the curious Galois groups of type $(2,3)$ are of level $\leq 70$, hence, will be in the LMFDB. We use the LMFDB for this search. We find that if $\operatorname{H}$ is a curious Galois group of type $(2,3)$, then the level of $\operatorname{H}$ is equal to $24$.

The candidates of level $4 \cdot 9$ are the groups \texttt{36.36.1.1} and \texttt{36.24.0.10}. 

The candidates of level $2 \cdot 9$ are the groups \texttt{18.24.0.1} and \texttt{18.36.0.1}.

The candidates of level $16 \cdot 3$ are \texttt{48.72.0.1}, \texttt{48.72.0.2}, \texttt{48.72.0.3}, and \texttt{48.72.0.4}.

The candidates of level $8 \cdot 3$ are

\texttt{24.18.0.6}, \texttt{24.18.0.5}, \texttt{24.36.0.15}, \texttt{24.36.0.16}, \texttt{24.36.0.17}, \texttt{24.36.0.18}, \href{https://lmfdb.org/ModularCurve/Q/24.6.1.2/}{\texttt{24.6.1.2}},

\href{https://lmfdb.org/ModularCurve/Q/24.18.1.8/}{\texttt{24.18.1.8}}, \href{https://lmfdb.org/ModularCurve/Q/24.18.1.5/}{\texttt{24.18.1.5}}, \texttt{24.36.1.61}, \texttt{24.36.1.62}, \texttt{24.48.1.79}, \texttt{24.8.0.9}, 

\texttt{24.8.0.10}, \texttt{24.24.0.99}, \texttt{24.24.0.100},
\texttt{24.24.0.101}, \texttt{24.24.0.102}, \href{https://lmfdb.org/ModularCurve/Q/24.12.1.3/}{\texttt{24.12.1.3}}, 

\href{https://lmfdb.org/ModularCurve/Q/24.36.1.8/}{\texttt{24.36.1.8}}, \href{https://lmfdb.org/ModularCurve/Q/24.36.1.3/}{\texttt{24.36.1.3}}, \href{https://lmfdb.org/ModularCurve/Q/24.24.1.2/}{\texttt{24.24.1.2}}, \texttt{24.72.1.4}, \texttt{24.72.1.1}.

The candidates of level $4 \cdot 3$ are

\texttt{12.12.0.5}, \texttt{12.18.0.1}, \texttt{12.18.0.2}, \texttt{12.8.0.5}, \texttt{12.24.0.5}, \texttt{12.24.0.4}, and \texttt{12.24.0.3}.

The candidates of level $2 \cdot 3$ are

\texttt{6.8.0.1}, \texttt{6.9.0.1}, \texttt{6.12.0.1}, \texttt{6.18.0.1}, \texttt{6.18.0.2}, \texttt{6.36.0.1}, and \texttt{6.24.0.1}.

Note that each of the remaining groups of genus $1$ are of level $24$ or the group \texttt{36.36.1.1}. In the case of \texttt{36.36.1.1}, there are no proper subgroups that are arithmetically admissible and have genus less than $2$. By a pigeonhole principle argument, \texttt{36.36.1.1} is not a curious Galois group. We move on to the case that the level of the group in question is equal to $24$.

We refer the reader to Table \ref{Curious Galois groups} which contains the groups of interest of level $24$.
We write the label of the generic elliptic curve generated by a group $\operatorname{H}$ in the third column of Table \ref{Curious Galois groups}. When the LMFDB has an example of an elliptic curve $E/\QQ$ such that $\overline{\rho}_{E,24}(G_{\QQ})$ is conjugate to $\operatorname{H}$, we write the label of $E$ in the second column of Table \ref{Curious Galois groups}.
We write the words \texttt{NOT CURIOUS} in the second column of Table \ref{Curious Galois groups}, if the corresponding group $\operatorname{H}$ has no proper arithmetically admissible subgroups of genus less than $2$ but there is no corresponding elliptic curve $E/\QQ$ in the LMFDB such that $\overline{\rho}_{E,24}(G_{\QQ})$ is conjugate to $\operatorname{H}$. The remaining seven groups in Table \ref{Curious Galois groups} are curious and we write the word \texttt{CURIOUS} in the second column of Table \ref{Curious Galois groups}. We prove that the remaining seven groups of level $24$ are in fact curious.

\begin{itemize}

\item \href{https://lmfdb.org/ModularCurve/Q/24.6.1.2/}{\texttt{24.6.1.2}}

The group \href{https://lmfdb.org/ModularCurve/Q/24.6.1.2/}{\texttt{24.6.1.2}} was proven to be curious in \cite{Daniels2018SerresCO}.

\item \href{https://lmfdb.org/ModularCurve/Q/24.12.1.3/}{\texttt{24.12.1.3}}

The group \href{https://lmfdb.org/ModularCurve/Q/24.12.1.3/}{\texttt{24.12.1.3}} generates a modular curve that is isomorphic to the elliptic curve $\operatorname{X}$ with label \texttt{576.e1}. Note that $\operatorname{X}(\QQ)_{\text{tors}} \cong \ZZ \times \ZZ / 2 \ZZ$. There are four arithmetically admissible subgroups of \href{https://lmfdb.org/ModularCurve/Q/24.12.1.3/}{\texttt{24.12.1.3}} of index $2$ and genus $1$, two of which generate the elliptic curve $\operatorname{X}'$ with label \texttt{576.e3}. Note that $\operatorname{X}'(\QQ) \cong \ZZ \times \ZZ / 2 \ZZ$. We list the two subgroups of index $2$ that generate \texttt{576.3} below.

Let $g_{1} = (-14,-8)$ be a generator of $\operatorname{X}(\QQ)$ of infinite order and let $g_{1}' = (10,28)$ be a generator of $\operatorname{X}'(\QQ)$ of infinite order. Let $g_{2} = (-12,0)$ be the rational point on $\operatorname{X}$ of order $2$ and let $g_{2}' = (6,0)$ be the rational point on $\operatorname{X}'$ of order $2$. Then there is an isogeny
$$\phi \colon \operatorname{X}' \to \operatorname{X}$$
generated by $g_{2}'$ such that $\phi(g_{1}') = -2 \cdot g_{1}$.
\begin{enumerate}
    \item $\operatorname{H}_{1} =$ \href{https://lmfdb.org/ModularCurve/Q/24.24.1.2/}{\texttt{24.24.1.2}} is a subgroup of \href{https://lmfdb.org/ModularCurve/Q/24.12.1.3/}{\texttt{24.12.1.3}} of index $2$ which is the direct product of \texttt{8.2.0.2} and \texttt{3.12.0.1} and generates an elliptic curve $\operatorname{X}_{1}$.
    \item $\operatorname{H}_{2} =$ \texttt{24.24.1.129} is a subgroup of \href{https://lmfdb.org/ModularCurve/Q/24.12.1.3/}{\texttt{24.12.1.3}} of index $2$ and generates an elliptic curve $\operatorname{X}_{2}$. Note that the reduction of \texttt{24.24.1.129} modulo $8$ is conjugate to \texttt{8.2.0.2} and the reduction modulo $3$ is conjugate to \texttt{3.12.0.1}.
\end{enumerate}
The rational points on $\operatorname{X}$ are of the form $[A] \cdot g_{1} + [B] \cdot g_{2}$ where $A$ is an integer and $B = 0$ or $1$. Note that for an integer $N$, the \textit{j}-invariant corresponding to $[A] \cdot g_{1}$ is equal to the \textit{j}-invariant corresponding to $[A] \cdot g_{1} + g_{2}$. To see why this is true, we take the \textit{j}-invariant corresponding to \href{https://lmfdb.org/ModularCurve/Q/24.12.1.3/}{\texttt{24.12.1.3}} from the LMFDB and see that all powers of $y$ are even. Thus, we substitute $y^{2}(X)$ as $X^{3} - 540X - 4752$ with $X$ as $x$. Next, we note that $g_{2} = (-12,0)$ and thus, we substitute $y^{2}(X)$ as $X^{3}-540X - 4752$ with $X = \frac{y^{2}(X)}{(X^{2}-(-12))^{2}} - X - (-12)$, in other words, $X$ is the $x$-coordinate of the sum of a rational point $(x,y)$ on $\operatorname{X}$ and $(-12,0)$. We observe those two \textit{j}-invariants are equal.

Let $E_{1}/\QQ$ be the elliptic curve \texttt{350.f3}. Then $\overline{\rho}_{E_{1},24}(G_{\QQ})$ is conjugate to \texttt{24.72.1.4}, the only subgroup of $\operatorname{H}_{1}$ of index $3$. Let $E_{2}/\QQ$ be the elliptic curve with label \texttt{65280.g1}. Then $\overline{\rho}_{E{2},24}(G_{\QQ})$ is conjugate to \texttt{24.72.1.50}, the only subgroup of $\operatorname{H}_{2}$ of index $3$. Moreover, \texttt{24.72.1.4} is not conjugate to \texttt{24.72.1.50}. We partition $S = \operatorname{X}(\QQ)$ into two subsets: $S_{1}$ the subset of $S$ of rational points $[A] \cdot g_{1} + [B] \cdot g_{2}$ where $A$ is even and $B$ is $0$ or $1$ and $S_{2}$ the subset of $S$ of rational points $[A] \cdot g_{1} + [B] \cdot g_{2}$ where $A$ is odd and $B = 0$ or $1$. Let $i = 1$ or $2$. The rational morphism $\psi_{i} \colon \operatorname{X}_{i} \to \operatorname{X}$ is of the form $\tau_{D_{i}} \circ [M_{i}] \circ \phi$ where $[M_{i}]$ is the identity map or the inversion map and $D_{i} = [A_{i}] \cdot g_{1} + [B_{i}] \cdot g_{2}$ for some integer $A_{i}$ and $B_{i} = 0$ or $1$. As we have found elliptic curves $E_{i}$ such that $\overline{\rho}_{E,24}(G_{\QQ})$ is conjugate to a subgroup of $\operatorname{H}_{i}$ but not to a subgroup of $\operatorname{H}_{j}$ when $i \neq j$, for all $i,j \in \left\{1,2\right\}$ we have that $A_{i} \not \equiv A_{j} \bmod 2$ when $i \neq j$. Thus, $\operatorname{H}$ is a curious Galois group. Thus, \href{https://lmfdb.org/ModularCurve/Q/24.12.1.3/}{\texttt{24.12.1.3}} is a curious Galois group.

\item \href{https://lmfdb.org/ModularCurve/Q/24.24.1.2/}{\texttt{24.24.1.2}}

The group \href{https://lmfdb.org/ModularCurve/Q/24.24.1.2/}{\texttt{24.24.1.2}} was proven to be curious in \cite{Chiloyan20232adicGI}.

\item \href{https://lmfdb.org/ModularCurve/Q/24.18.1.5/}{\texttt{24.18.1.5}}, the direct product of \texttt{8.6.0.1} and \texttt{3.3.0.1}, which generates the elliptic curve $\operatorname{X}/\QQ$ with label \texttt{576.e4}. Note that $\operatorname{X}(\QQ) \cong \ZZ \times \ZZ / 2 \ZZ$. There are eight proper subgroups of \href{https://lmfdb.org/ModularCurve/Q/24.18.1.5/}{\texttt{24.18.1.5}} that are arithmetically admissible and have genus $0$ or $1$; four of which have index $2$ and four of which have index $4$. We list them below. The index-$4$ subgroups of \href{https://lmfdb.org/ModularCurve/Q/24.18.1.5/}{\texttt{24.18.1.5}} of genus $1$ are index-$2$ subgroups of \href{https://lmfdb.org/ModularCurve/Q/24.36.1.8/}{\texttt{24.36.1.8}}. The modular curves generated by \href{https://lmfdb.org/ModularCurve/Q/24.36.1.8/}{\texttt{24.36.1.8}} and \texttt{24.36.1.149} are isomorphic to the elliptic curve $\operatorname{X}'/\QQ$ with label \texttt{576.e2}. Note that $\operatorname{X}'(\QQ) \cong \ZZ \times \ZZ / 2 \ZZ$.

Let $g_{1} = (2,-4)$ be a generator of $\operatorname{X}(\QQ)$ of infinite order and let $g_{1}' = (10,-24)$ be a generator of $\operatorname{X}'(\QQ)$ of infinite order. Let $g_{2} = (-2,0)$ be the rational point on $\operatorname{X}$ of order $2$ and let $g_{2}' = (4,0)$ be the rational point on $\operatorname{X}'$ of order $2$. There is an isogeny
        $$\phi \colon \operatorname{X}' \to \operatorname{X}$$
        generated by $g_{2}'$ such that $\phi(g_{1}') = g_{1}$.
\begin{enumerate}
    \item $\operatorname{H}_{1}$ = \href{https://lmfdb.org/ModularCurve/Q/24.36.1.8/}{\texttt{24.36.1.8}} (also curious) is a subgroup of \href{https://lmfdb.org/ModularCurve/Q/24.18.1.5/}{\texttt{24.18.1.5}} of index $2$ which is the direct product of \texttt{8.6.0.1} and \texttt{3.6.0.1} and generates the elliptic curve $\operatorname{X}_{1}$.
    \item \texttt{24.36.1.103} is a subgroup of \href{https://lmfdb.org/ModularCurve/Q/24.18.1.5/}{\texttt{24.18.1.5}} of index $2$. The modular curve generated by \texttt{24.36.1.103} does not have points defined over $\QQ_{2}$.
    \item \texttt{24.36.1.114} is a subgroup of \href{https://lmfdb.org/ModularCurve/Q/24.18.1.5/}{\texttt{24.18.1.5}} of index $2$. The modular curve generated by \texttt{24.36.1.114} has no points defined over $\QQ_{2}$.
    \item $\operatorname{H}_{2} =$ \texttt{24.36.1.149} is a subggroup of \href{https://lmfdb.org/ModularCurve/Q/24.18.1.5/}{\texttt{24.18.1.5}} of index $2$ and generates the elliptic curve $\operatorname{X}_{2}$. The reduction of \texttt{24.36.1.149} modulo $8$ is conjugate to \texttt{8.6.0.1} and the reduction modulo $3$ is conjugate to \texttt{3.6.0.1}.
\item \texttt{24.72.1.1}, a subgroup of \href{https://lmfdb.org/ModularCurve/Q/24.18.1.5/}{\texttt{24.18.1.5}} of index $4$.
\item \texttt{24.72.1.32} is a subgroup of \href{https://lmfdb.org/ModularCurve/Q/24.18.1.5/}{\texttt{24.18.1.5}} of index $4$.
\item \texttt{24.72.1.46} is a subgroup of \href{https://lmfdb.org/ModularCurve/Q/24.18.1.5/}{\texttt{24.18.1.5}} of index $4$.
\item \texttt{24.72.1.60} is a subgroup of \href{https://lmfdb.org/ModularCurve/Q/24.18.1.5/}{\texttt{24.18.1.5}} of index $4$.
\end{enumerate}
The rational points on $\operatorname{X}$ are of the form $[A] \cdot g_{1} + [B] \cdot g_{2}$ where $A$ is an integer and $B = 0$ or $1$. We partition $S = \operatorname{X}(\QQ)$ into two subsets: $S_{1}$ the subset of $S$ where $B = 0$ and $S_{2}$ the subset of $S$ where $B = 1$.

Let $E/\QQ$ be the elliptic curve with label \texttt{5888.a2}. Then $\overline{\rho}_{E,24}(G_{\QQ})$ is conjugate to \texttt{24.36.1.149}. Let $E/\QQ$ be the elliptic curve with label \texttt{350.f6}. Then $\overline{\rho}_{E,24}(G_{\QQ})$ is conjugate to \texttt{24.72.1.1}, a subgroup of \href{https://lmfdb.org/ModularCurve/Q/24.36.1.8/}{\texttt{24.36.1.8}} of index $2$. Note that the genus of a proper arithmetically admissible subgroup of \texttt{24.36.1.149} is greater than $1$ and hence, \texttt{24.72.1.1} is not conjugate to a subgroup of \texttt{24.36.1.149}.

Let $i = 1$ or $2$. The rational morphism $\psi_{i} \colon \operatorname{X}_{i} \to \operatorname{X}$ is of the form $\tau_{D_{i}} \circ [M_{i}] \circ \phi$ where $[M_{i}]$ is the identity map or the inversion map and $D_{i} = [A_{i}] \cdot g_{1} + [B_{i}] \cdot g_{2}$ for some integer $A_{i}$ and $B_{i} = 0$ or $1$. As we have found elliptic curves $E_{i}$ such that $\overline{\rho}_{E,24}(G_{\QQ})$ is conjugate to a subgroup of $\operatorname{H}_{i}$ but not to a subgroup of $\operatorname{H}_{j}$ when $i \neq j$ for all $i,j \in \left\{1,2\right\}$, we have that $B_{i} \not \equiv B_{j} \bmod 2$ when $i \neq j$. Thus, \href{https://lmfdb.org/ModularCurve/Q/24.18.1.5/}{\texttt{24.18.1.5}} is a curious Galois group.

\item \href{https://lmfdb.org/ModularCurve/Q/24.18.1.8/}{\texttt{24.18.1.8}}, the direct product of \texttt{8.6.0.6} and \texttt{3.3.0.1}, which generates the elliptic curve $\operatorname{X}/\QQ$ with label \texttt{576.e4}. Note that $\operatorname{X}(\QQ) \cong \ZZ \times \ZZ / 2 \ZZ$. There are eight proper subgroups of \href{https://lmfdb.org/ModularCurve/Q/24.18.1.8/}{\texttt{24.18.1.8}} that are arithmetically admissible and have genus $1$; four of which have index $2$ and four of which have index $4$. The admissible subgroups of \texttt{24.18.18} of index $4$ and genus $1$ are subgroups of \href{https://lmfdb.org/ModularCurve/Q/24.36.1.3/}{\texttt{24.36.1.3}} of index $2$. We list them below. The modular curves generated by \href{https://lmfdb.org/ModularCurve/Q/24.36.1.3/}{\texttt{24.36.1.3}} and \texttt{24.36.1.146} are isomorphic to the elliptic curve $\operatorname{X}'/\QQ$ with label \texttt{576.e2}. Note that $\operatorname{X}'(\QQ) \cong \ZZ \times \ZZ / 2 \ZZ$.

Let $g_{1} = (2,-4)$ be a generator of $\operatorname{X}(\QQ)$ of infinite order and let $g_{1}' = (10,-24)$ be a generator of $\operatorname{X}'(\QQ)$ of infinite order. Let $g_{2} = (-2,0)$ be the rational point on $\operatorname{X}$ of order $2$ and let $g_{2}' = (4,0)$ be the rational point on $\operatorname{X}'$ of order $2$. There is an isogeny
        $$\phi \colon \operatorname{X}' \to \operatorname{X}$$
        generated by $g_{2}'$ such that $\phi(g_{1}') = g_{1}$.

\begin{enumerate}
    \item $\operatorname{H}_{1} =$ \href{https://lmfdb.org/ModularCurve/Q/24.36.1.3/}{\texttt{24.36.1.3}}, (also curious) a subgroup of \href{https://lmfdb.org/ModularCurve/Q/24.18.1.8/}{\texttt{24.18.1.8}} of index $2$ which is the direct product of \texttt{8.6.0.6} and \texttt{3.6.0.1} and generates a modular curve $\operatorname{X}_{1}$.
    \item \texttt{24.36.1.27} a subgroup of \href{https://lmfdb.org/ModularCurve/Q/24.18.1.8/}{\texttt{24.18.1.8}} of index $2$. The modular curve generated by \texttt{24.36.1.27} does not contain points defined over $\QQ_{2}$.
    \item \texttt{24.36.1.142}, a subgroup of \href{https://lmfdb.org/ModularCurve/Q/24.18.1.8/}{\texttt{24.18.1.8}} of index $2$. The modular curve generated by \texttt{24.36.1.27} does not contain points defined over $\QQ_{2}$.
    \item $\operatorname{H}_{2} =$ \texttt{24.36.1.146}, a subgroup of \href{https://lmfdb.org/ModularCurve/Q/24.18.1.8/}{\texttt{24.18.1.8}} of index $2$ that generates a modular curve $\operatorname{X}_{2}$. The reduction of \texttt{24.36.1.146} modulo $8$ is conjugate to \texttt{8.6.0.6} and the reduction modulo $3$ is conjugate to \texttt{3.3.0.1}.
    \item \texttt{24.72.1.4} a subgroup of \href{https://lmfdb.org/ModularCurve/Q/24.18.1.8/}{\texttt{24.18.1.8}} of index $4$.
    \item \texttt{24.72.1.10}, a subgroup of \href{https://lmfdb.org/ModularCurve/Q/24.18.1.8/}{\texttt{24.18.1.8}} of index $4$.
    \item \texttt{24.72.1.50} a subgroup of \href{https://lmfdb.org/ModularCurve/Q/24.18.1.8/}{\texttt{24.18.1.8}} of index $4$.
    \item \texttt{24.72.1.58} a subgroup of \href{https://lmfdb.org/ModularCurve/Q/24.18.1.8/}{\texttt{24.18.1.8}} of index $4$.
\end{enumerate}
The rational points on $\operatorname{X}$ are of the form $[A] \cdot g_{1} + [B] \cdot g_{2}$ where $A$ is an integer and $B = 0$ or $1$. We partition $S = \operatorname{X}(\QQ)$ into two subsets: $S_{1}$ the subset of $S$ where $B = 0$ and $S_{2}$ the subset of $S$ where $B = 1$. All proper, arithmetically admissible subgroups of \texttt{24.36.1.146} have genus greater than $1$. By the pidgeonhole principle, \texttt{24.36.1.146} is not a curious Galois group. Let $E/\QQ$ be the elliptic curve \texttt{350.f3}. Then $\overline{\rho}_{E,24}(G_{\QQ})$ is conjugate to \texttt{24.72.1.4}, a subgroup of \href{https://lmfdb.org/ModularCurve/Q/24.36.1.3/}{\texttt{24.36.1.3}} of index $2$. Note that \href{https://lmfdb.org/ModularCurve/Q/24.36.1.3/}{\texttt{24.36.1.3}} is a group of genus $1$ and hence, cannot be conjugate to a subgroup of \texttt{24.36.1.146}.

Let $i = 1$ or $2$. The rational morphism $\psi_{i} \colon \operatorname{X}_{i} \to \operatorname{X}$ is of the form $\tau_{D_{i}} \circ [M_{i}] \circ \phi$ where $[M_{i}]$ is the identity map or the inversion map and $D_{i} = [A_{i}] \cdot g_{1} + [B_{i}] \cdot g_{2}$ for some integer $A_{i}$ and $B_{i} = 0$ or $1$. As there are elliptic curves $E_{i}$ such that $\overline{\rho}_{E_{i},24}(G_{\QQ})$ is conjugate to a subgroup of $\operatorname{H}_{i}$ but not to a subgroup of $\operatorname{H}_{j}$ when $i \neq j$ for all $i,j \in \left\{1,2\right\}$, we have that $B_{i} \not \equiv B_{j} \bmod 2$ when $i \neq j$. Thus, \href{https://lmfdb.org/ModularCurve/Q/24.18.1.8/}{\texttt{24.18.1.8}} is a curious Galois group.

\item \href{https://lmfdb.org/ModularCurve/Q/24.36.1.3/}{\texttt{24.36.1.3}}, the direct product of \texttt{8.6.0.6} and \texttt{3.6.0.1}, which generates the elliptic curve $\operatorname{X}/\QQ$ with label \texttt{576.e2}. Note that $\operatorname{X}(\QQ) \cong \ZZ \times \ZZ / 2 \ZZ$. There are four proper subgroups of \href{https://lmfdb.org/ModularCurve/Q/24.36.1.3/}{\texttt{24.36.1.3}} that are arithmetically admissible and have genus $0$ or $1$. We list them below. The modular curves generated by \texttt{24.72.1.4} and \texttt{24.72.1.50} are isomorphic to the elliptic curve $\operatorname{X}'/\QQ$ with label \texttt{576.e4}. Note that $\operatorname{X}'(\QQ) \cong \ZZ \times \ZZ / 2 \ZZ$.

Let $g_{1} = (10,-24)$ be a generator of $\operatorname{X}(\QQ)$ of infinite order and let $g'_{1} = (2,-4)$ be a generator of $\operatorname{X}'(\QQ)$ of infinite order. Let $g_{2} = (4,0)$ be the rational point on $\operatorname{X}$ of order $2$. There is an isogeny
        $$\phi \colon \operatorname{X}' \to \operatorname{X}$$
        generated by the rational point on $\operatorname{X}'$ of order $2$ such that $\phi(g_{1}') = 2 \cdot g_{1}$.

\begin{enumerate}
    \item $\operatorname{H}_{1} = \texttt{24.72.1.4}$ is a subgroup of \href{https://lmfdb.org/ModularCurve/Q/24.36.1.3/}{\texttt{24.36.1.3}} of index $2$ which is the direct product of \texttt{8.6.0.6} and \texttt{3.12.0.1} and generates an elliptic curve $\operatorname{X}_{1}$.
    \item \texttt{24.72.1.10} is a subgroup of \href{https://lmfdb.org/ModularCurve/Q/24.36.1.3/}{\texttt{24.36.1.3}} of index $2$. The modular curve generated by \texttt{24.72.1.10} does not have points defined over $\QQ_{2}$.
    \item $\operatorname{H}_{2} = \texttt{24.72.1.50}$ is a subgroup of \href{https://lmfdb.org/ModularCurve/Q/24.36.1.3/}{\texttt{24.36.1.3}} of index $2$ which generates an elliptic curve $\operatorname{X}_{2}$. The reduction of \texttt{24.72.1.50} modulo $8$ is conjugate to \texttt{8.6.0.6} and the reduction modulo $3$ is conjugate to \texttt{3.6.0.1}.
    \item \texttt{24.72.1.58} is a subgroup of \href{https://lmfdb.org/ModularCurve/Q/24.36.1.3/}{\texttt{24.36.1.3}} of index $2$. The modular curve generated by \texttt{24.72.1.58} does not have points defined over $\QQ_{2}$.
    \end{enumerate}

The rational points on $\operatorname{X}$ are of the form $[A] \cdot g_{1} + [B] \cdot g_{2}$ where $A$ is an integer and $B = 0$ or $1$. We note that for an integer $A$, the \textit{j}-invariant associated to $[A] \cdot g_{1}$ is equal to the \textit{j}-invariant associated to $[A] \cdot g_{1} + g_{2}$. To see why this is true, we take the \textit{j}-invariant corresponding to \href{https://lmfdb.org/ModularCurve/Q/24.36.1.3/}{\texttt{24.36.1.3}} from the LMFDB and see that all powers of $y$ are even. Thus, we substitute $y^{2}(X)$ as $X^{3} - 60X + 176$ with $X$ as $x$. Next, we note that $g_{2} = (4,0)$ and thus, we substitute $y^{2}(X)$ as $X^{3} - 60X + 176$ with $X = \frac{y^{2}(X)}{(X^{2}-4)^{2}} - X - 4$, in other words, $X$ is the $x$-coordinate of the sum of a rational point $(x,y)$ on $\operatorname{X}$ and $(4,0)$. We observe those two \textit{j}-invariants are equal.

We partition $S = \operatorname{X}(\QQ)$ into two subsets: $S_{1}$ the subset of $S$ where $A$ is even and $S_{2}$ the subset of $S$ where $A$ is odd. Let $E_{1}/\QQ$ be the elliptic curve \texttt{350.f3}. Then $\overline{\rho}_{E_{1},24}(G_{\QQ})$ is conjugate to \texttt{24.72.1.4}. Let $E_{2}/\QQ$ be the elliptic curve \texttt{65280.g1}. Then $\overline{\rho}_{E_{2},24}(G_{\QQ})$ is conjugate to \texttt{24.72.1.50}. Note that \texttt{24.72.1.4} is not conjugate to a subgroup of $\operatorname{H}_{2}$ and \texttt{24.72.1.50} is not conjugate to a subgroup of $\operatorname{H}_{1}$.

For $i = 1$ and $2$, the rational morphism $\psi_{i} \colon \operatorname{X}_{i} \to \operatorname{X}$ is of the form $\tau_{D_{i}} \circ [M_{j}] \circ \phi$ where $[M_{j}]$  is the inversion map or the identity map and $D_{i} = [A_{i}] \cdot g_{1} + [B_{i}] \cdot g_{2}$ is a rational point on $\operatorname{X}$. As we have found elliptic curves $E_{1}$ and $E_{2}$ such that $\overline{\rho}_{E_{1},24}(\QQ)$ is conjugate to a subgroup of $\operatorname{H}_{1}$ but not to $\operatorname{H}_{2}$ and $\overline{\rho}_{E_{2},24}(\QQ)$ is conjugate to a subgroup of $\operatorname{H}_{2}$ but not to $\operatorname{H}_{1}$, we have that $A_{1} \not \equiv A_{2} \mod 2$. Thus, \href{https://lmfdb.org/ModularCurve/Q/24.36.1.3/}{\texttt{24.36.1.3}} is a curious Galois group.

\item \href{https://lmfdb.org/ModularCurve/Q/24.36.1.8/}{\texttt{24.36.1.8}}, the direct product of \texttt{8.6.0.1} and \texttt{3.6.0.1}, which generates the elliptic curve $\operatorname{X}/\QQ$ with label \texttt{576.e2}. Note that $\operatorname{X}(\QQ) \cong \ZZ \times \ZZ / 2 \ZZ$. There are four proper subgroups of \href{https://lmfdb.org/ModularCurve/Q/24.36.1.8/}{\texttt{24.36.1.8}} that are arithmetically admissible and have genus $0$ or $1$. We list them below. The modular curves generated by \texttt{24.72.1.1} and \texttt{24.72.1.46} are isomorphic to the elliptic curve $\operatorname{X}'$ with label \texttt{576.e4}. Note that $\operatorname{X}'(\QQ) \cong \ZZ \times \ZZ / 2 \ZZ$. 

Let $g_{1} = (10,-24)$ be a generator of $\operatorname{X}(\QQ)$ of infinite order and let $g'_{1} = (2,-4)$ be a generator of $\operatorname{X}'(\QQ)$ of infinite order. Let $g_{2} = (4,0)$ be the rational point on $\operatorname{X}$ of order $2$ and let $A$ be an integer. There is an isogeny
        $$\phi \colon \operatorname{X}' \to \operatorname{X}$$
        generated by the rational point on $\operatorname{X}'$ of order $2$ such that $\phi(g') = 2 \cdot g$.

\begin{enumerate}
\item $\operatorname{H}_{1} = \texttt{24.72.1.1}$ is an index-$2$ subgroup of \href{https://lmfdb.org/ModularCurve/Q/24.36.1.8/}{\texttt{24.36.1.8}} which is the direct product of \texttt{8.6.0.1} and \texttt{3.12.0.1} and generates a modular curve $\operatorname{X}_{1}$.
    \item \texttt{24.72.1.32} is an index-$2$ subgroup of \href{https://lmfdb.org/ModularCurve/Q/24.36.1.8/}{\texttt{24.36.1.8}}. The modular curve generated by \texttt{24.72.1.32} has no points defined over $\QQ_{2}$
\item $\operatorname{H}_{2} = \texttt{24.72.1.46}$ is an index-$2$ subgroup of \href{https://lmfdb.org/ModularCurve/Q/24.36.1.8/}{\texttt{24.36.1.8}} which generates an elliptic curve $\operatorname{X}_{2}$. The reduction of \texttt{24.72.1.46} modulo $8$ is conjugate to \texttt{8.6.0.1} and the reduction modulo $3$ is conjugate to \texttt{3.6.0.1}.
\item \texttt{24.72.1.60} is an index-$2$ subgroup of \href{https://lmfdb.org/ModularCurve/Q/24.36.1.8/}{\texttt{24.36.1.8}}. The modular curve generated by \texttt{24.72.1.60} does not have points defined over $\QQ_{2}$.
    \end{enumerate}
The rational points on $\operatorname{X}$ are of the form $[A] \cdot g_{1} + [B] \cdot g_{2}$ where $A$ is an integer and $B = 0$ or $1$. We note that for an integer $A$, the \textit{j}-invariant associated to $[A] \cdot g_{1}$ is equal to the \textit{j}-invariant associated to $[A] \cdot g_{1} + g_{2}$. To see why this is true, we take the \textit{j}-invariant corresponding to \href{https://lmfdb.org/ModularCurve/Q/24.36.1.8/}{\texttt{24.36.1.8}} from the LMFDB and see that all powers of $y$ are even. Thus, we substitute $y^{2}(X)$ as $X^{3} - 60X + 176$ with $X$ as $x$. Next, we note that $g_{2} = (4,0)$ and thus, we substitute $y^{2}(X)$ as $X^{3} - 60X + 176$ with $X = \frac{y^{2}(X)}{(X^{2}-4)^{2}} - X - 4$, in other words, $X$ is the $x$-coordinate of the sum of a rational point $(x,y)$ on $\operatorname{X}$ and $(4,0)$. We observe those two \textit{j}-invariants are equal.

We partition $S = \operatorname{X}(\QQ)$ into two subsets: $S_{1}$ the subset of $S$ where $A$ is even and $S_{2}$ the subset of $S$ where $A$ is odd. Let $E_{1}/\QQ$ be the elliptic curve \texttt{350.f6}. Then $\overline{\rho}_{E_{1},24}(G_{\QQ})$ is conjugate to \texttt{24.72.1.1}. Let $E_{2}/\QQ$ be the elliptic curve with \texttt{65280.g2}. Then $\overline{\rho}_{E_{2},24}(G_{\QQ})$ is conjugate to \texttt{24.72.1.46}. Note that \texttt{24.72.1.1} is not conjugate to \texttt{24.72.1.46}. For $i = 1$ and $2$, the rational morphism $\psi_{i} \colon \operatorname{X}_{i} \to \operatorname{X}$ is of the form $\tau_{D_{i}} \circ [M_{i}] \circ \phi$ where $[M_{i}]$ is the inversion map or the identity map and $D_{i} = [A_{i}] \cdot g_{1} + [B_{i}] \cdot g_{2}$ for some integer $A_{i}$ and $B_{i} = 0$ or $1$. As we have found elliptic curves $E_{1}$ and $E_{2}$ such that $\overline{\rho}_{E_{1},24}(G_{\QQ})$ is conjugate to a subgroup of $\operatorname{H}_{1}$ but not to a subgroup of $\operatorname{H}_{2}$ and $\overline{\rho}_{E_{2},24}(G_{\QQ})$ is conjugate to a subgroup of $\operatorname{H}_{2}$ but not to a subgroup of $\operatorname{H}_{1}$, we have that $A_{1} \not \equiv A_{2} \mod 2$. Thus, \href{https://lmfdb.org/ModularCurve/Q/24.36.1.8/}{\texttt{24.36.1.8}} is a curious Galois group.

\end{itemize}
\end{enumerate}
\end{proof}

\begin{corollary}
    Let $m$ be an integer greater than equal to $2$ and let $p_{2} < \ldots < p_{m}$ be odd primes. Then there are no curious Galois groups of type $(2,p_{2}, \ldots, p_{m})$.
\end{corollary}

\begin{proof}
    All that needs to be proven is the case when $m = 3$. Let $\operatorname{H}_{2}$ be a proper subgroup of $\operatorname{GL}(2, \ZZ_{2})$. Assume by way of contradiction that $\operatorname{H}_{p_{i}}$ is a proper subgroup of $\operatorname{GL}(2, \ZZ_{p_{i}})$ for $i = 2, 3$ such that $\operatorname{H} = \operatorname{H}_{2} \times \operatorname{H}_{p_{2}} \times \operatorname{H}_{p_{3}}$ is a curious Galois group. Then the genus of $\operatorname{H}$ is equal to $0$ or $1$. By Theorem \ref{odd theorem} and its proof, $p_{2} = 3$ and $\operatorname{H}_{3} = \texttt{3.3.0.1}$. Moreover, $p_{3} = 5$ or $p_{3} = 7$. We break up the proof into two cases.

    \begin{enumerate}

        \item $p_{2} = 3$ and $p_{3} = 5$

        In this case, a priori, we have that $\operatorname{H}_{2}$ is one of \texttt{2.3.0.1}, \texttt{4.4.0.1}, \texttt{8.2.0.1}, \texttt{8.2.0.2}, \texttt{8.6.0.2}, \texttt{8.6.0.4}, or \texttt{8.6.0.6}, $\operatorname{H}_{3}$ is \texttt{3.3.0.1}, and $\operatorname{H}_{5}$ is one of \texttt{5.5.0.1}, \texttt{5.6.0.1}, \texttt{5.10.0.1}, or \texttt{5.15.0.1}. Moreover, in each case, the genus of $\operatorname{H} = \operatorname{H}_{2} \times \operatorname{H}_{3} \times \operatorname{H}_{5}$ is greater than or equal to $2$. Hence, $\operatorname{H}$ is not a curious Galois group.
        \item $p_{2} = 3$ and $p_{3} = 7$

        In this case, a priori, we have that $\operatorname{H}_{7} = \texttt{7.21.0.1}$. Let $\operatorname{H}_{2}$ be a group of $2$-power level from \cite{SZ}. Then the genus of $\operatorname{H}_{2} \times \operatorname{H}_{7}$ is greater than or equal to $2$. Hence, $\operatorname{H}$ is not a curious Galois group.
    \end{enumerate}
\end{proof}

\newpage

\section{Appendix}

        \begin{table}[h!]
	\renewcommand{\arraystretch}{1.25}
	\begin{tabular}{|c|c|c|}
		\hline
		Group & Example elliptic curve & Generic Elliptic curve \\
		\hline
            \texttt{20.20.1.2} & \texttt{12996.a1} & \texttt{400.a1} \\
            \hline
            \href{https://lmfdb.org/ModularCurve/Q/40.12.1.5/}{\texttt{40.12.1.5}} & \texttt{CURIOUS} & \texttt{320.a2} \\
            \hline
            \href{https://lmfdb.org/ModularCurve/Q/40.20.1.2/}{\texttt{40.20.1.2}} & \texttt{CURIOUS} & \texttt{1600.c4} \\
            \hline
            \texttt{40.24.1.17} & \texttt{450.c3} & \texttt{320.a1} \\
            \hline
            \texttt{40.24.1.22} & \texttt{450.c1} & \texttt{320.a1} \\
            \hline
            \href{https://lmfdb.org/ModularCurve/Q/40.36.1.2/}{\texttt{40.36.1.2}} & \texttt{CURIOUS} & \texttt{320.c2} \\
            \hline
            \href{https://lmfdb.org/ModularCurve/Q/40.36.1.4/}{\texttt{40.36.1.4}} & \texttt{CURIOUS} & \texttt{320.a4} \\
            \hline
            \href{https://lmfdb.org/ModularCurve/Q/40.36.1.5/}{\texttt{40.36.1.5}} & \texttt{CURIOUS} & \texttt{320.a4} \\
            \hline
            \texttt{40.72.1.15} & \texttt{198.c4} & \texttt{320.a3} \\
            \hline
            \texttt{40.72.1.22} & \texttt{198.c3} & \texttt{320.a3} \\
            \hline
            \texttt{40.72.1.24} & \texttt{198.c2} & \texttt{320.a3} \\
            \hline
            \texttt{40.72.1.29} & \texttt{198.c1} & \texttt{320.a3}\\
            \hline
\href{https://lmfdb.org/ModularCurve/Q/24.6.1.2/}{\texttt{24.6.1.2}} & \texttt{CURIOUS} & \texttt{576.e3} \\
\hline
\href{https://lmfdb.org/ModularCurve/Q/24.12.1.3/}{\texttt{24.12.1.3}} & \texttt{CURIOUS} & \texttt{576.e1} \\
\hline
\href{https://lmfdb.org/ModularCurve/Q/24.18.1.5/}{\texttt{24.18.1.5}} & \texttt{CURIOUS} & \texttt{576.e4} \\
\hline
\href{https://lmfdb.org/ModularCurve/Q/24.18.1.8/}{\texttt{24.18.1.8}} & \texttt{CURIOUS} & \texttt{576.e4} \\
\hline
\href{https://lmfdb.org/ModularCurve/Q/24.24.1.2/}{\texttt{24.24.1.2}} & \texttt{CURIOUS} & \texttt{576.e3} \\
\hline
\href{https://lmfdb.org/ModularCurve/Q/24.36.1.3/}{\texttt{24.36.1.3}} & \texttt{CURIOUS} & \texttt{576.e2} \\
\hline
\href{https://lmfdb.org/ModularCurve/Q/24.36.1.8/}{\texttt{24.36.1.8}} & \texttt{CURIOUS} & \texttt{576.e2} \\
\hline
\texttt{24.36.1.61} & \texttt{126350.t1} & \texttt{576.e4} \\
\hline
\texttt{24.36.1.62} & \texttt{126350.t2} & \texttt{576.e4} \\
\hline
\texttt{24.48.1.79} & \texttt{NOT CURIOUS} & \texttt{576.e4} \\
\hline
\texttt{24.72.1.1} & \texttt{350.f6} & \texttt{576.e4} \\
\hline
\texttt{24.72.1.4} & \texttt{350.f3} & \texttt{576.e4} \\
\hline
\end{tabular}
	\caption{Curious Galois groups}
	\label{Curious Galois groups}
\end{table}

\bibliography{bibliography}
\bibliographystyle{plain}

\end{document}